\documentclass[reqno]{amsart}

\usepackage{amsmath}
\usepackage{amssymb}
\usepackage{mathtools}
\usepackage{suffix}
\usepackage{bm}
\usepackage{a4wide}
\usepackage{nicefrac}
\usepackage{graphicx}
\usepackage[small,it]{caption}
\usepackage{subfig}
\usepackage{algpseudocode}
\usepackage{enumerate}
\usepackage{color}
\usepackage{cite}

\newcommand{\real}{{\mathbb{R}}}
\newcommand{\nat}{{\mathbb{N}}}
\newcommand{\Hp}[2][]{H^{#2}_{#1}}
\newcommand{\Lp}[1][2]{L_{#1}}
\newcommand{\Cp}[1][0]{C^{#1}}

\newcommand{\mesh}[1][h]{\mathcal{T}_{#1}}
\newcommand{\cgspace}{V_{\sf FEM}}
\newcommand{\cgspacei}[1][i]{V_{{\sf FEM},#1}}
\renewcommand{\vec}[1]{\bm{#1}}

\newcommand{\F}{\mathsf F}
\newcommand{\wt}{\widetilde}

\newcommand{\dx}{\, \mathrm{d}\vec{x}}
\newcommand{\ds}{\, \mathrm{d} s}

\newcommand{\e}{\mathrm{e}}
\newcommand{\jl}{\left[\!\left[}
\newcommand{\jr}{\right]\!\right]}
\newcommand{\jmp}[1]{\jl#1\jr}

\newcommand{\order}[1]{\mathcal{O}\left(#1\right)}
\newcommand{\massmat}{\bm{M}}

\DeclarePairedDelimiter{\abs}{\lvert}{\rvert}
\DeclarePairedDelimiter{\norm}{\lVert}{\rVert}

\newcommand{\tnorm}[2][]{#1\lvert\!#1\lvert\!#1\lvert #2 #1\rvert\!#1\rvert\!#1\rvert_\Omega}
\WithSuffix\newcommand\tnorm*[1]{\left\lvert\!\left\lvert\!\left\lvert #1 \right\rvert\!\right\rvert\!\right\rvert_\Omega}

\newcommand{\emptyarg}{\:\cdot\:}

\newtheorem{algorithm}[equation]{Algorithm}

\newtheorem{Remark}[equation]{Remark}
\newenvironment{remark}{\begin{Remark}\rm}{\end{Remark}}
\newtheorem{theorem}[equation]{Theorem}
\newtheorem{proposition}[equation]{Proposition}

\newtheorem{lemma}[equation]{Lemma}

\numberwithin{equation}{section}

\newcommand{\secref}[1]{Section~\ref{#1}}

\newcommand{\thref}[1]{Theorem~\ref{#1}}

\newcommand{\subfigref}[2]{Figure~\ref{#1}\subref{#1:#2}}

\begin{document}

\title[Iterative Galerkin Discretizations for Strongly Monotone Problems]{Iterative Galerkin Discretizations\\ for Strongly Monotone Problems}

\author[S.~Congreve]{Scott Congreve}
\address{%
    Fakult\"at f\"ur Mathematik\\
    Universit\"at Wien\\
    Oskar-Morganstern-Platz 1\\
    A-1090 Wien\\
    Austria%
}
\email{scott.congreve@univie.ac.at}

\author[T.~P.~Wihler]{Thomas P. Wihler}
\address{%
    Mathematisches Institut\\
    Universit\"at Bern\\
    Sidlerstrasse 5\\
    CH-3012 Bern\\
    Switzerland%
}
\email{wihler@math.unibe.ch}

\thanks{Both authors acknowledge the support of the Swiss National Science Foundation (SNF)}

\begin{abstract}
In this article we investigate the use of fixed point iterations to solve the Galerkin approximation of strictly monotone problems. As opposed to Newton's method, which requires information from the previous iteration in order to linearise the iteration matrix (and thereby to recompute it) in each step, the alternative method used in this article exploits the monotonicity properties of the problem, and only needs the iteration matrix calculated once for all iterations of the fixed point method. We outline the abstract \emph{a priori} and \emph{a posteriori} analysis for the iteratively obtained solutions, and apply this to a finite element approximation of a second-order elliptic quasilinear boundary value problem.  We show both theoretically, as well as numerically, how the number of iterations of the fixed point method can be restricted in dependence of the mesh size, or of the polynomial degree, to obtain optimal convergence. Using the \emph{a posteriori} error analysis we also devise an adaptive algorithm for the generation of a sequence of Galerkin spaces (adaptively refined finite element meshes in the concrete example) to minimise the number of iterations on each space.
\end{abstract}



\maketitle

\section{Introduction}
In this paper we study Galerkin approximations of strictly monotone problems of the form:
\begin{equation}\label{eq:P}
\text{find }u\in X:\qquad A(u,v)=0\quad\forall v\in X.
\end{equation}
Here, $X$ is a real Hilbert space, with inner product denoted by~$(\cdot,\cdot)_X$ and induced norm~$\|x\|=\sqrt{(x,x)_X}$. Furthermore, $A:\,X\times X\to\mathbb{R}$ is a possibly nonlinear form such that, for any~$u\in X$, the mapping~$v\mapsto A(u,v)$ is linear and bounded. Moreover, we suppose that~$A$ satisfies
\begin{enumerate}[(P1)]
\item the \emph{strong monotonicity} property
\begin{equation}\label{eq:P1}\tag{P1}
A(u,u-v)-A(v,u-v)\ge c_0\|u-v\|_X^2\qquad\forall u,v\in X,
\end{equation}
for a constant~$c_0>0$, and
\item the \emph{Lipschitz continuity} condition
\begin{equation}\label{eq:P2}\tag{P2}
|A(u,w)-A(v,w)|\le L\|u-v\|_X\|w\|_X\qquad\forall u,v,w\in X,
\end{equation}
with a constant~$L>0$. 
\end{enumerate}
Under these assumptions, there exists a unique solution~$u\in X$ of the weak formulation~\eqref{eq:P}; see, e.g., \cite[Theorem~2.H]{Zeidler} or~\cite{necas}. In addition, the solution can be obtained as limit of a sequence $u^0, u^1, u^2,\ldots\in X$ resulting from the fixed point iteration
\begin{equation}\label{eq:it}
(u^n,v)_X=(u^{n-1},v)_X-\frac{c_0}{L^2} A(u^{n-1},v)\quad \forall v\in X,\qquad n\ge 1,
\end{equation}
for an arbitrary initial value~$u^0\in X$. Indeed, defining the contraction constant
\begin{equation}\label{eq:k}
k=\sqrt{1-\left(\frac{c_0}{L}\right)^2},
\end{equation}
there holds the \emph{a priori} bound
\begin{equation}\label{eq:apriori_exact}
\|u-u^n\|_X\le\frac{k^n}{1-k}\|u^0-u^1\|_X,\qquad n\ge 1,
\end{equation}
for the iteration~\eqref{eq:it}, i.e., $\|u-u^n\|_X\xrightarrow{n\to\infty} 0$. 

Restricting the iteration~\eqref{eq:it} to a finite dimensional linear subspace~$X_h\subseteq X$, leads to an iterative Galerkin approximation scheme for~\eqref{eq:P}. More precisely, we consider, for an initial guess~$u^0_h\in X_h$ and $n\ge 1$, the iteration
\begin{equation}\label{eq:FPI}
u_h^n\in X_h:\qquad (u^{n}_h,v)_X=(u_h^{n-1},v)_X-\frac{c_0}{L^2}A\left(u_h^{n-1},v\right)\qquad\forall v\in X_h,
\end{equation}
where~$c_0$ and~$L$ are the constants from~\eqref{eq:P1} and~\eqref{eq:P2} respectively. We emphasize that the problem of finding~$u^n_h$ from~$u^{n-1}_h$ in the iteration scheme~\eqref{eq:FPI} is \emph{linear} and uniquely solvable. Similarly as for~\eqref{eq:it} and~\eqref{eq:P}, the fixed point iteration~\eqref{eq:FPI} converges to the (unique) solution~$u_h\in X_h$ of the Galerkin formulation
\begin{equation}\label{eq:Galerkin}
A(u_h,v)=0\qquad\forall v\in X_h.
\end{equation}
Furthermore, we note the \emph{a priori} bound
\begin{equation}\label{eq:FPapriori}
\|u_h-u_h^{n}\|_{X}\le\frac{k^n}{1-k}\|u_h^{0}-u_h^{1}\|_{X},\qquad n\ge 1,
\end{equation}
analogous to~\eqref{eq:apriori_exact}.

In solving nonlinear differential equations numerically two approaches are commonly employed. Either the nonlinear problem under consideration is discretized by means of a suitable numerical scheme thereby resulting in a (finite-dimensional) nonlinear algebraic system, or the differential equation problem is approximated by a sequence of (locally) linearized problems which are discretized subsequently. The latter approach is attractive from both a computational as well as an analytical view point; indeed, working with a sequence of linear problems allows the application of linear solvers as well as the use of a linear numerical analysis framework (e.g., in deriving error estimates). In the context of fixed point linearizations~\eqref{eq:FPI} yet another benefit comes into play: solving for~$u^n_h$ from~$u_h^{n-1}$ involves setting up and inverting a mass matrix on the left-hand side of~\eqref{eq:FPI}. We emphasize that this matrix is the same for all iterations; hence, it only needs to be computed once (on a given Galerkin space).

The idea of approximating nonlinear problems within a \emph{linear} Galerkin framework has been applied in a variety of works. For example in the article~\cite{ChaillouSuri:07}, the authors have considered general linearizations of strongly monotone operators, and have derived computable estimators for the total error (consisting of the linearization error and the Galerkin approximation error), with identifiable components for each of the error sources. A more specific linearization approach for monotone problems, which is based on the Newton method, has been presented in~\cite{El-AlaouiErnVohralik:11}. In a related context linear finite element approximations resulting from adaptive Newton linearization techniques as applied to semilinear problems have been investigated in the papers~\cite{AmreinWihler:14,AmreinWihler:15}. Finally, we remark that the linear Galerkin approximation approach for nonlinear problems is not only employed for the purpose of obtaining linearized schemes, but also to address the issue of modelling errors in linearized models; see, e.g.~\cite{ChaillouSuri:06,Han:94}.

The aim of the current paper is to derive \emph{a priori} and \emph{a posteriori} error bounds for the Galerkin iteration method~\eqref{eq:FPI}. Our error estimates are expressed as the summation of the linearization error resulting from the fixed point formulation with the Galerkin approximation error. In particular, based on the \emph{a posteriori} error analysis, we will develop an adaptive solution procedure for the numerical solution of~\eqref{eq:P} that features an appropriate interplay between the fixed point iterations and possible Galerkin space enrichments (e.g., mesh refinements for finite elements); specifically, our scheme selects between these two options depending on whichever constitutes the dominant part of the total error. In this way, we aim to keep the number of fixed point iterations at a minimum in the sense that no unnecessary iterations are performed if they are not expected to contribute a substantial reduction of the error on the actual Galerkin space.

The outline of the rest of this article is as follows. In \secref{sc:abstract} we derive an abstract analysis for the fixed point iteration~\eqref{eq:FPI}, which includes the derivation of both \emph{a priori} and \emph{a posteriori} error bounds; in addition, we formulate an abstract adaptive procedure. The purpose of Section~\ref{sc:PDE} is the application of our abstract theory to the finite element approximation of a second-order elliptic quasi-linear elliptic diffusion reaction boundary value problem; in particular, we derive a fully adaptive algorithm based on suitable \emph{a posteriori} error estimates, and provide a series of numerical experiments. Finally, in \secref{sec:conclusion} we summarise the work presented and draw some conclusions.

\section{Abstract analysis}\label{sc:abstract}

\subsection{Fixed point Galerkin approximation}
As previously discussed, we let~$X_h$ be a finite dimensional linear subspace of a Hilbert space~$X$. Then, in order to approximate the solution $u\in X$ of~\eqref{eq:P}, we consider the Galerkin solution~$u_h\in X_h$ defined in~\eqref{eq:Galerkin}. For the purpose of calculating~$u_h$ we consider, in turn, the discrete and linear fixed point iteration scheme~\eqref{eq:FPI}. Evidently, this is equivalent to a linear algebraic system of equations. More precisely, using basis functions $\phi_i\in X_h$, for $i=1,\dots,m$, where $m=\dim(X_h)$ is the number of degrees of freedom, and letting
    \[
        u_h^{n} = \sum_{i=1}^{m} \phi_i \alpha_i^{n},
    \]
    for some unknown coefficients~$\bm\alpha^n=\{\alpha_i^{n}\}_{i=1}^m$,
    we obtain the linear system version of the fixed point iteration~\eqref{eq:FPI}:
    \[
        \sum_{i=1}^m \massmat_{ij} \alpha_i^{n} = \sum_{i=1}^m \massmat_{ij} \alpha_i^{n-1} - \frac{c_{0}}{L^2} \vec{A}(\vec{\alpha}^{n-1})_{j}, \qquad \textrm{for } j=1,\dots,m.
    \]
  Here, $\massmat$ is the iteration matrix defined by $\massmat_{ij} = \left(\phi_i,\phi_j\right)_X$, and $\vec{A}(\vec{\alpha}^{n-1})$, with
  \[
  \vec{A}(\vec{\alpha}^{n-1})_j=A\left(\sum_{i=1}^{m} \phi_i \alpha_i^{n-1},\phi_j\right),\qquad j=1,\ldots,m,
  \]
   is the vector form of $A(u_h,v)$.
    We can see that the iteration matrix $\massmat$ does {\em not depend on the iteration number $n$}; hence, it only needs calculating once for all iterations of the fixed point method (on a given Galerkin space).

\subsection{\emph{A priori} error bound}

Denoting by
\begin{equation}\label{eq:error}
e_h^n=u-u_h^n
\end{equation} 
the error between the solution~$u$ of~\eqref{eq:P} and~$u^n_h$ from~\eqref{eq:FPI}, we employ the triangle inequality and~\eqref{eq:FPapriori} to obtain
\[
\|e_h^n\|_X\le\|u-u_h\|_X+\frac{k^n}{1-k}\|u_h^{0}-u_h^{1}\|_{X},
\]
where~$u_h\in X_h$ is the Galerkin solution defined in~\eqref{eq:Galerkin}. Furthermore, employing the monotonicity property~\eqref{eq:P1} leads to
\begin{align*}
c_0\|u-u_h\|_X^2
&\le A(u,u-u_h)-A(u_h,u-u_h)\\
&=A(u,u-v)-A(u_h,u-v),
\end{align*}
for any~$v\in X_h$. Involving~\eqref{eq:P2}, we conclude
\[
c_0\|u-u_h\|_X^2\le L\|u-u_h\|_X\|u-v\|_X\qquad\forall v\in X_h,
\]
and thus
\[
\|u-u_h\|_X\le \frac{L}{c_0}\|u-v\|_X\qquad\forall v\in X_h.
\]

Combining these estimates we obtain the following result.

\begin{proposition}\label{pr:apriori}
For the error between the solution~$u\in X$ of~\eqref{eq:P} and its iterative Galerkin approximation~$u_h^n\in X_h$ from~\eqref{eq:FPI} there holds the \emph{a priori} error estimate
\[
\|u-u_h^{n}\|_X\le \frac{L}{c_0}\inf_{v\in X_h}\|u-v\|_X+\frac{k^n}{1-k}\|u_h^{0}-u_h^{1}\|_{X},
\]
for any~$n\ge 1$.
\end{proposition}

\subsection{\emph{A posteriori} error analysis}
\label{sc:aposteriori}

In order to derive an \emph{a posteriori} error analysis for~\eqref{eq:FPI} let us consider the auxiliary problem of finding~$\widetilde{u}^n\in X$ such that
\begin{equation}\label{eq:auxiliary}
(\widetilde{u}^n,v)_X=(u_h^{n-1},v)_X-\frac{c_0}{L^2}A(u_h^{n-1},v)\quad\forall v\in X,\qquad n\ge 1.
\end{equation}
We note that $\widetilde{u}^n\in X$ is a \emph{reconstruction} (cf.~\cite{MakridakisNochetto:03}) in the sense that~$u^n_h\in X_h$ from~\eqref{eq:FPI} is the Galerkin approximation of~$\widetilde{u}^n$. We assume that we can bound the error between the solution~$\widetilde{u}^n\in X$ of~\eqref{eq:auxiliary} and its Galerkin approximation~$u_h^n\in X_h$ in terms of an \emph{a posteriori} computable quantity~$\eta(u_h^n,X_h)$, i.e.,
\begin{equation}\label{eq:eta}
\|\widetilde{u}^n-u_h^n\|_X\le \eta(u_h^n,X_h),\qquad n\ge 1.
\end{equation}
Involving the monotonicity property~\eqref{eq:P1}, the error~$e_h^n$ from~\eqref{eq:error} satisfies
\[
c_0\|e_h^n\|_X^2\le A(u,e_h^n)-A(u_h^n,e_h^n)=-A(u_h^n,e_h^n).
\]
Furthermore, recalling~\eqref{eq:auxiliary}, we write
\begin{align*}
c_0\|e_h^n\|_X^2
&\le A(u_h^{n-1},e_h^n)-A(u_h^n,e_h^n)+\frac{L^2}{c_0}(\widetilde u^n-u_h^{n-1},e_h^n)_X\\
&=A(u_h^{n-1},e_h^n)-A(u_h^n,e_h^n)+\frac{L^2}{c_0}(u_h^n-u_h^{n-1},e_h^n)_X+\frac{L^2}{c_0}(\widetilde u^n-u_h^n,e_h^n)_X.
\end{align*}
Then, using~\eqref{eq:P2} and applying the Cauchy-Schwarz inequality, we infer that
\begin{align*}
c_0\|e_h^n\|_X^2
&\le L\|u_h^{n-1}-u_h^n\|_X\|e_h^n\|_X+\frac{L^2}{c_0}\|u_h^n-u_h^{n-1}\|_X\|e_h^n\|_X +\frac{L^2}{c_0}\|\widetilde u^n-u_h^n\|_X\|e_h^n\|_X,
\end{align*}
and dividing by~$c_0\|e_h^n\|_X$, we obtain
\[
\|e_h^n\|_X\le \frac{L}{c_0}\left(1+\frac{L}{c_0}\right)\|u_h^n-u_h^{n-1}\|_X+\frac{L^2}{c_0^2}\|\widetilde u^n-u_h^n\|_X.
\]
Hence, inserting~\eqref{eq:eta}, the following result can be deduced.

\begin{proposition}\label{pr:aposteriori}
For the error between the solution~$u\in X$ of~\eqref{eq:P} and its iterative Galerkin approximation~$u_h^n\in X_h$ from~\eqref{eq:FPI} there holds the \emph{a posteriori} error estimate
\begin{equation}\label{eq:aposteriori}
\|e_h^n\|_X\le \frac{L}{c_0}\left(1+\frac{L}{c_0}\right)\|u_h^n-u_h^{n-1}\|_X+\frac{L^2}{c_0^2}\eta(u_h^n,X_h),
\end{equation}
where~$\eta(u_h^n,X_h)$ is given in~\eqref{eq:eta}.
\end{proposition}

\subsection{An abstract adaptive algorithm}\label{sc:adaptive}
The \emph{a posteriori} error estimate~\eqref{eq:aposteriori} shows that the error~$e_h^n$ from~\eqref{eq:error} is controlled by two separate parts: a fixed point iteration error given by $Lc_0^{-1}\left(1+Lc_0^{-1}\right)\|u_h^n-u_h^{n-1}\|_X$, and a Galerkin approximation term~$L^2c_0^{-2}\eta(u_h^n,X_h)$. When performing the fixed point iteration~\eqref{eq:FPI} it is worth noting that once the fixed point error is less than the Galerkin error carrying out another iteration will not cause a substantial reduction of the error on the actual Galerkin space. Based on this observation we are able to develop an algorithm which generates a sequence of hierarchically enriched Galerkin spaces $X_{h,1}\subset X_{h,2}\subset X_{h,3}\subset\ldots\subset X$, with the aim of performing a minimal number of fixed point iterations at each enrichment step. Our algorithm will, therefore, feature an interplay between fixed point iterations and Galerkin space refinements.

On the Galerkin space~$X_{h,i}$, $i\ge 1$, we define the \emph{Galerkin approximation error} by
\[
    \mathcal{E}_{{\sf Galerkin},i}^{n} := \frac{L^2}{c_0^2}\eta(u_{h,i}^n,X_{h,i}),
\]
and the \emph{fixed point error}
\[
    \mathcal{E}_{{\sf FP},i}^{n} := \frac{L}{c_0}\left(1+\frac{L}{c_0}\right)\|u_{h,i}^n-u_{h,i}^{n-1}\|_X.
\]
This allows us to write the \emph{a posteriori} error bound~\eqref{eq:aposteriori} as
\[
    \|u-u_{h,i}^{n}\|_X \leq \mathcal{E}_{{\sf Galerkin},i}^{n} + \mathcal{E}_{{\sf FP},i}^{n}.
\]
Here, we denote by~$u_{h,i}^{n}\in X_{h,i}$ the Galerkin solution obtained after $n$ steps of the fixed point iteration \eqref{eq:FPI} on the current space~$X_{h,i}$; for~$i>0$, the initial guess $u_{h,i}^{0}\in X_{h,i}$ on the current Galerkin space~$X_{h,i}$ is obtained as the natural inclusion (or a projection) of the solution~$u_{h,i-1}^{n^\star}\in X_{h,i-1}$ of the last (namely, the~$n^\star$-th) iteration on the previous Galerkin space~$X_{h,i-1}$ to the space $X_{h,i}$. In particular, the fixed point iteration index~$n$ is reinitiated in each space enrichment step.

\begin{algorithm}\label{algo:refinement_abstract}
Choose an initial starting space~$X_{h,0}$, and an initial guess $u_{h,0}^{0}\in X_{h,0}$.
    \begin{algorithmic}
        	\For{$i \gets 0,1,2,\dots$}
	        \State $n \gets 0$
	        \Repeat
	            \State $n \gets n+1$
	            \State Perform a single fixed point iteration \eqref{eq:FPI} to calculate $u_{h,i}^{n}\in X_{h,i}$.
	        \Until{$\mathcal{E}_{{\sf FP},i}^{n} \leq \vartheta \mathcal{E}_{{\sf Galerkin},i}^{n}$}
	        \State Perform a hierarchical enrichment of~$X_{h,i}$ 	
	        \State\quad based on the error indicator~$\eta(u_{h,i}^n,X_{h,i})$ from~\eqref{eq:eta} 
	        \State\quad to obtain a new Galerkin space~$X_{h,i+1}\supset X_{h,i}$. 

	        \State Define $u_{h,i+1}^{0} \gets u_{h,i}^{n}$ (by inclusion~$X_{h,i+1}\hookleftarrow X_{h,i}$ or by projection)
	    \EndFor
	\end{algorithmic}
	Here, $\vartheta>0$ is a prescribed parameter. The algorithm is stopped if either the iteration number~$i$ reaches a given maximum, or if the right-hand side of~\eqref{eq:aposteriori} is found to be sufficiently small.
\end{algorithm}

\section{Application to quasilinear elliptic PDE}\label{sc:PDE}

\subsection{Problem formulation}
In this section, we focus on the numerical approximation of second-order elliptic diffusion reaction boundary value problems of the form
\begin{alignat}{2}
    -\nabla\cdot\left( \mu(\vec{x}, \vert \nabla u \vert)\nabla u\right) + f(\vec{x}, u) &= 0 &\qquad& \text{in } \Omega, \label{eqn:nonlinear PDE} \\
   u &= 0 & & \text{on } \Gamma, \label{eqn:nonlinear PDE BC}
\end{alignat}
where $\Omega$ is a bounded, open, polygonal domain in $\real^2$, with
boundary $\Gamma=\partial\Omega$. Here, we assume the following {\em monotonicity} conditions on the nonlinearities $\mu$ and $f$:
\begin{enumerate}
\item $\mu \in \Cp(\overline{\Omega} \times [0,\infty))$, and there exist constants
    $\alpha_1 \geq \alpha_2 > 0$ such that the following property is satisfied:
    \begin{equation}
        \alpha_2(t-s)\leq \mu(\vec{x},t)t-\mu(\vec{x},s)s \leq \alpha_1 (t-s), \qquad t\geq s\geq0, \quad \vec{x}\in\overline{\Omega}. \label{eqn:nonlinearity assump bound}
    \end{equation}
\item $f \in \Cp(\overline{\Omega} \times \real)$, and there exist constants
    $\beta_1 \geq \beta_2 \geq 0$ such that
    \begin{equation}
        \beta_2(t-s)\leq f(\vec{x},t)-f(\vec{x},s) \leq \beta_1 (t-s), \qquad t\geq s, \quad \vec{x}\in\overline{\Omega}. \label{eqn:forcing-func assump bound}
    \end{equation}
\end{enumerate}
From \cite[Lemma 2.1]{Liu1994} we note that, as $\mu$ satisfies
\eqref{eqn:nonlinearity assump bound}, for all vectors
$\vec{v},\vec{w}\in\real^2$ and all $\vec{x}\in\overline{\Omega}$, we have
\begin{align}
    \vert \mu(\vec{x},\vert\vec{v}\vert)\vec{v}-\mu(\vec{x},\vert\vec{w}\vert)\vec{w}\vert &\leq \alpha_1\vert\vec{v}-\vec{w}\vert, \label{eqn:liu result 1} \\
    \alpha_2 {\vert\vec{v}-\vec{w}\vert}^2 &\leq (\mu(\vec{x},\vert\vec{v}\vert)\vec{v}-\mu(\vec{x},\vert\vec{w}\vert)\vec{w})\cdot(\vec{v}-\vec{w}).  \label{eqn:liu result 2}
\end{align}
Similarly, as $f$ satisfies \eqref{eqn:forcing-func assump bound}, it holds that for all $s,t\in\real$ and all $\vec{x}\in\overline{\Omega}$,
\begin{align}
    \abs{ f(\vec{x},t)-f(\vec{x},s) } &\leq \beta_1 \abs{t-s}, \label{eqn:forcing-func bound 1} \\
    \beta_2 \abs{t-s}^2 &\leq (f(\vec{x},t)-f(\vec{x},s))(t-s).  \label{eqn:forcing-func bound 2}
\end{align}
For ease of notation we shall suppress the dependence of $\mu$ and $f$ on
$\vec{x}$ and write $\mu(t)$ and $f(u)$ instead of $\mu(\vec{x},t)$ and $f(\vec{x},u)$, respectively.

The weak formulation of the boundary value problem
\eqref{eqn:nonlinear PDE}--\eqref{eqn:nonlinear PDE BC} is to find $u\in X:=\Hp[0]{1}(\Omega)$ such that
\begin{equation}\label{eqn:weak formulation}
    A(u,v) = 0 \qquad\forall v\in\Hp[0]{1}(\Omega),
\end{equation}
where
\begin{equation}\label{eq:A}
    A(u,v) = \int_\Omega \left\{ \mu(\vert \nabla u \vert)\nabla u \cdot \nabla v + f(u)v \right\} \dx,\qquad  u, v \in \Hp[0]{1}(\Omega).
\end{equation}
Throughout this section, we use function spaces based on a polygonal
Lipschitz domain $D\subset\real^2$. We denote by $\Hp{k}(D)$ the Sobolev space of order $k\in\nat_0$ endowed with the norm $\norm{\emptyarg}_{\Hp{k}(D)}$. In the case that
$k=0$, we set $\Hp{k}(D)=\Lp(D)$ and denote the matching norm by $\norm{\emptyarg}_{\Lp(D)}$.
Furthermore, we define $\Hp[0]{1}(D)$ as the space of functions in $\Hp{1}(D)$ with zero trace on~$\partial D$.

Introducing the inner product
\begin{equation*}
    (u, v)_\Omega \coloneqq \int_\Omega \left\{ \alpha_2 \nabla u\cdot \nabla v + \beta_2 u v \right\} \dx,\qquad  u, v \in \Hp[0]{1}(\Omega),
\end{equation*}
where~$\alpha_2$ and~$\beta_2$ are the constants from~\eqref{eqn:nonlinearity assump bound} and~\eqref{eqn:forcing-func assump bound}, respectively, we note the induced norm
\[
    \tnorm{v}^2 \coloneqq \alpha_2 \norm{\nabla v}_{\Lp(\Omega)}^2 + \beta_2 \norm{v}_{\Lp(\Omega)}^2
\]
on~$\Hp[0]{1}(\Omega)$.

\subsection{Basic properties} Under the conditions~\eqref{eqn:nonlinearity assump bound} and~\eqref{eqn:forcing-func assump bound} we can show that the properties~\eqref{eq:P1} and~\eqref{eq:P2} are satisfied for~$X=H^1_0(\Omega)$, and~$\|\cdot\|_X:=\tnorm{\cdot}$. Indeed, noting the Poincar\'e inequality,
\begin{equation}\label{eq:CP}
\norm{v}_{\Lp(\Omega)}\le C_P\norm{\nabla v}_{\Lp(\Omega)}\qquad\forall v\in \Hp[0]{1}(\Omega),
\end{equation}
where~$C_P>0$ is a constant dependent only on~$\Omega$, there holds the ensuing result.

\begin{proposition}
    \label{pr:monotone}
    Provided that \eqref{eqn:nonlinearity assump bound} and~\eqref{eqn:forcing-func assump bound} hold, then the form~$A$ from~\eqref{eq:A} is both strongly monotone with constant~$c_0=1$ in~\eqref{eq:P1}, and Lipschitz continuous with constant
    \begin{equation}\label{eq:L}
1\le    L = \frac{\alpha_1+\max\left(\beta_1,\nicefrac{\alpha_1\beta_2}{\alpha_2}\right) C_P^2}{\alpha_2+\beta_2C_P^2}
    \end{equation}
in~\eqref{eq:P2}. Here, $C_P>0$ is the Poincar\'e constant from~\eqref{eq:CP}.
    \end{proposition}

\begin{proof}
    In order to show~\eqref{eq:P1} with~$c_0=1$, we apply~\eqref{eqn:liu result 2} and~\eqref{eqn:forcing-func bound 2} to arrive at
    \begin{align*}
        A(u,u-v)-A(v,u-v)
             &=\int_\Omega  \left( \mu(\abs{\nabla u})\nabla u - \mu(\abs{\nabla v}) \nabla v \right) \cdot \nabla (u-v) \dx \\
                &\quad+ \int_\Omega \left( f(u)-f(v) \right) (u-v) \dx\\
            &\geq \int_\Omega \left\{ \alpha_2 \abs{\nabla(u-v)}^2 + \beta_2 \abs{u-v}^2 \right \} \dx
            = \tnorm{u-v}^2.
    \end{align*}

Furthermore, to prove the Lipschitz continuity property~\eqref{eq:P2}, we recall \eqref{eqn:liu result 1} and \eqref{eqn:forcing-func bound 1}. In combination with the Cauchy-Schwarz inequality, this yields
    \begin{align*}
        \abs{ A(u,w)-A(v,w) } &\leq \int_\Omega \abs{ \mu(\abs{\nabla u})\nabla u - \mu(\abs{\nabla v}) \nabla v }\abs{\nabla w} \dx
                + \int_\Omega \abs{ f(u)-f(v) } \abs{w} \dx \\
            &\leq \alpha_1 \norm{\nabla(u-v)}_{\Lp(\Omega)} \norm{\nabla w}_{\Lp(\Omega)} + \beta_1 \norm{u-v}_{\Lp(\Omega)} \norm{w}_{\Lp(\Omega)}.
    \end{align*}
    We first consider the case when $\beta_2 = 0$; then, noting that $\tnorm{v} =\sqrt{\alpha_2} \norm{\nabla v}_{\Lp(\Omega)}$,
    we apply the Poincar\'e inequality~\eqref{eq:CP} to observe that
    \begin{align*}
        \abs{ A(u,w)-A(v,w) } &\leq \alpha_1 \norm{\nabla(u-v)}_{\Lp(\Omega)} \norm{\nabla w}_{\Lp(\Omega)} + \beta_1 C_P^2 \norm{\nabla(u-v)}_{\Lp(\Omega)} \norm{\nabla w}_{\Lp(\Omega)} \\
            &= \left( \frac{\alpha_1+\beta_1 C_P^2}{\alpha_2} \right) \tnorm{u-v}\tnorm{w}.
    \end{align*}
    When $\beta_2 > 0$ we introduce a constant $0\le\delta\le\beta_1$ and apply the Poincar\'e inequality~\eqref{eq:CP}, to get that
    \begin{align*}
        \abs{ A(u,w)-A(v,w) }
            &\leq
                \alpha_1 \norm{\nabla(u-v)}_{\Lp(\Omega)} \norm{\nabla w}_{\Lp(\Omega)} + (\beta_1 - \delta) \norm{u-v}_{\Lp(\Omega)} \norm{w}_{\Lp(\Omega)} \\
                    &\quad+ \delta C_P^2 \norm{\nabla(u-v)}_{\Lp(\Omega)} \norm{\nabla w}_{\Lp(\Omega)}
             \\
            &=
                \left( \frac{\alpha_1+\delta C_P^2}{\alpha_2} \right) \sqrt{\alpha_2} \norm{\nabla(u-v)}_{\Lp(\Omega)} \sqrt{\alpha_2} \norm{\nabla w}_{\Lp(\Omega)} \\
                    &\quad+ \left( \frac{\beta_1 - \delta}{\beta_2} \right) \sqrt{\beta_2} \norm{u-v}_{\Lp(\Omega)} \sqrt{\beta_2} \norm{w}_{\Lp(\Omega)}.
    \end{align*}
Using the Cauchy-Schwarz inequality yields
\[
\abs{ A(u,w)-A(v,w) }\le L(\delta)\tnorm{u-v}\tnorm{w},
\]
where
\[
L(\delta)=\max\left(\frac{\alpha_1+\delta C_P^2}{\alpha_2},\frac{\beta_1 - \delta}{\beta_2}\right).
\]
Minimizing~$L(\delta)$ within the given range, $0\le\delta\le\beta_1$, depends on the constants~$\alpha_1,\alpha_2,\beta_1,\beta_2$. More precisely, if~$\nicefrac{\alpha_1}{\alpha_2}\ge\nicefrac{\beta_1}{\beta_2}$ then~$\delta^\star=0$ is the optimal choice, and there holds~$L(0)=\nicefrac{\alpha_1}{\alpha_2}\ge1$; otherwise, we select
\[
\delta^\star := \frac{\beta_1 \alpha_2 - \beta_2 \alpha_1}{\alpha_2 + \beta_2 C_P^2}\in(0,\beta_1),
\]
and thereby obtain
\[
L(\delta^\star)=\frac{\alpha_1+\beta_1 C_P^2}{\alpha_2+\beta_2C_P^2}\ge 1.
\]
This completes the proof.
\end{proof}

\begin{remark}
Incidentally, referring to, e.g., \cite[Theorem~3.3.23]{necas} or~\cite[Theorem~2.H]{Zeidler}, the above result, Proposition~\ref{pr:monotone}, guarantees the existence of a unique solution~$u\in \Hp[0]{1}(\Omega)$ of~\eqref{eqn:weak formulation}. 
\end{remark}

\begin{remark}
We note that the fixed point iteration~\eqref{eq:it} for the current problem~\eqref{eqn:nonlinear PDE}--\eqref{eqn:nonlinear PDE BC} reads in strong form as
\begin{alignat*}{2}
-\alpha_2\Delta u^n+\beta_2u^n
&=-\alpha_2\Delta u^{n-1}+\beta_2 u^{n-1}-L^{-2}\left(-\nabla\cdot(\mu(|\nabla u^{n-1}|)\nabla u^{n-1})+f(u^{n-1})\right)&\qquad&\text{in }\Omega\\
u^n&=0&&\text{on }\partial\Omega,
\end{alignat*}
in $\Hp{-1}(\Omega)$, the dual space of~$H^1_0(\Omega)$, for~$n\ge 1$.
\end{remark}

\subsection{Finite element discretization}
\label{sec:fe_form}
In order to solve~\eqref{eqn:weak formulation} by a fixed point Galerkin iteration, we will use a finite element framework.

\subsubsection{Meshes and spaces}
We consider regular and shape-regular meshes $\mesh$ that partition the domain $\Omega\subset\real^2$ into open disjoint triangles and/or parallelograms $K$, such that $\overline{\Omega} = \bigcup_{K\in\mesh} \overline{K}$. We denote by
$h_K$ the elemental diameter of $K\in\mesh$, and let $h=\max_{K\in\mesh} h_K$.

With this notation, for a fixed polynomial degree~$p\ge 1$, we are ready to introduce the finite element space
\begin{equation}\label{eq:FEMspace}
    \cgspace \coloneqq \{ v\in\Hp[0]{1}(\Omega) : v\vert_K\in\mathcal{S}_{p}(K)\ \forall K\in\mesh \} \subset \Hp[0]{1}(\Omega),
\end{equation}
where
\[
    \mathcal{S}_{p}(K) =
    \begin{cases}
        \mathcal{P}_{p}(K) &\text{if } K \text{ is a triangle}, \\
        \mathcal{Q}_{p}(K) &\text{if } K \text{ is a parallelogram}.
    \end{cases}
\]
Here, $\mathcal{P}_p(K)$ denotes the space of polynomials of total order at most $p$ on $K$, while
$\mathcal{Q}_p(K)$ is the tensored space of polynomials of order at most $p$ in each variable on $K$.

%
\subsubsection{Iterative Galerkin FEM}
\label{sec:fp_iteration}
Based on the class of spaces~$\cgspace$ introduced before, we can now introduce the finite element formulation for a {\em linear} fixed point formulation~\eqref{eq:FPI} of~\eqref{eqn:weak formulation}: Given
an initial guess $u_h^{0}\in\cgspace$, we iterate for~$n=1,2,3\ldots$,
\begin{equation}
    \label{eqn:fixed_point_itn}
    \left(u_h^{n}, v_h\right)_\Omega = \left(u_h^{n-1}, v_h\right)_\Omega - L^{-2}A(u_h^{n-1}, v_h)\qquad\forall v_h\in\cgspace.
\end{equation}

\begin{remark}\label{rm:epsilon}
Recalling~\eqref{eq:k} the contraction constant for this iteration is given by
 \begin{equation*}
         k = \sqrt{1 - L^{-2}} < 1.
    \end{equation*}
Here, we point out that, in the singularly perturbed case when~$\alpha_2\approx\alpha_1=\mathcal{O}(\varepsilon)$, for~$0<\varepsilon\ll 1$, and $\beta_2\approx\beta_1=\mathcal{O}(1)$, the contraction factor~$k$ does not deteriorate to~$1$. Indeed, this follows from the fact that the Lipschitz constant~$L$ from Proposition~\ref{pr:monotone} remains robustly bounded from~$0$ in this situation.
\end{remark}

\subsection{Error Analysis}

We will now apply the abstract analysis derived in Section~\ref{sc:abstract} to the iterative Galerkin method~\eqref{eqn:fixed_point_itn} for the numerical approximation of~\eqref{eqn:nonlinear PDE}--\eqref{eqn:nonlinear PDE BC}.

\subsubsection{\emph{A priori} error bound}
\label{sec:a_priori}
Using our abstract \emph{a priori} error analysis from Proposition~\ref{pr:apriori} and applying suitable $hp$-approximation results (see, e.g., \cite{BaSu87}), we obtain a bound for the error between the numerical solution $u_h^{n}$ obtained at the $n$-th step of the
fixed point iteration \eqref{eqn:fixed_point_itn} and of the exact solution $u$ from \eqref{eqn:weak formulation}. For simplicity of presentation we assume a (quasi-) uniform diameter $h>0$ of all elements.

\begin{theorem}
    \label{theorem:a priori fp}
    Let $u \in \Hp{\kappa+1}(\Omega)\cap\Hp[0]{1}(\Omega)$, with $\kappa\geq 1$, be the solution to the weak formulation~\eqref{eqn:weak formulation}, $u_h^{0}\in\cgspace$ any initial guess,
    and $u_h^{n}\in\cgspace$ the numerical
    solution after $n$ steps of the fixed point iteration \eqref{eqn:fixed_point_itn}; then, for $n\geq 1$, there holds the {\em a priori} error estimate
    \begin{equation}\label{eq:apriori}
        \tnorm*{u-u_h^{n}} \leq CL \frac{h^{\min(\kappa,p)}}{p^\kappa} \norm{u}_{\Hp{\kappa+1}(\Omega)}
            + 2L^2 \left( 1 - L^{-2} \right)^{\nicefrac{n}{2}} \tnorm*{u_h^{0}-u_h^{1}},
    \end{equation}
    where $C>0$ is a constant independent of $h$, $p$, and~$L$ from~\eqref{eq:L}, but depends on~$\alpha_2$ and~$\beta_2$ from~\eqref{eqn:nonlinearity assump bound} and~\eqref{eqn:forcing-func assump bound}, respectively.
\end{theorem}

\begin{remark}\label{rm:h-p-FEM}
From the above Theorem~\ref{theorem:a priori fp} it is possible to predict the (approximate) number of fixed point iterations required to obtain an optimal convergence rate in the linear finite element iteration~\eqref{eqn:fixed_point_itn}. To this end, we ask for the second term on the right-hand side of~\eqref{eq:apriori} to converge at least at the rate of the first term. In order to discuss the resulting convergence behaviour of the numerical solution~$u^{n}_h$ obtained from~\eqref{eqn:fixed_point_itn}, we distinguish two different cases:
\begin{itemize}
\item {\em $h$-FEM:} We fix a low polynomial degree~$p$ and investigate the convergence properties with respect to the mesh size~$h$ as $h\to 0$ (mesh refinement). Here, for~$\kappa\ge p$, we need
\[
(1-L^{-2})^{\nicefrac{n}{2}}=\order{h^p},
\]
and hence, $n=\order{|\log h|}$ as~$h\to 0$.
\item {\em $p$-FEM:} We now fix the mesh, and suppose that the solution of~\eqref{eqn:nonlinear PDE}--\eqref{eqn:nonlinear PDE BC} is analytic. Then, as~$p\to\infty$, it can be shown that the FEM converges exponentially (see~\cite{Schwab1998}), i.e., the error bound~\eqref{eq:apriori} reads
\[
\tnorm*{u-u_h^{n}}\le\order{e^{-bp}}+ 2 L^2 \left( 1 - L^{-2} \right)^{\nicefrac{n}{2}} \tnorm*{u_h^{0}-u_h^{1}},
\]
for some constant~$b>0$. Again, balancing the two terms on the right, we require~$n=\order{p}$ iterations as~$p\to\infty$.
\end{itemize}
We will test these observations with some numerical experiments in Section~\ref{sec:numerics}.
\end{remark}

\subsubsection{{\em A posteriori} error analysis}
In this section we obtain an \emph{a posteriori} error bound for the error between the numerical solution $u_h^{n}$ obtained at the $n$-th step of the
fixed point iteration \eqref{eqn:fixed_point_itn} and of the exact solution $u$ obtained from \eqref{eqn:nonlinear PDE}--\eqref{eqn:nonlinear PDE BC}. According to our analysis in Section~\ref{sc:aposteriori} the key is to derive an \emph{a posteriori} error estimate between the reconstruction~$\widetilde{u}^n\in\Hp[0]{1}(\Omega)$ from~\eqref{eq:auxiliary} and the iterative Galerkin solution~$u^n_h$ from~\eqref{eq:FPI} (i.e.,~$u_h^n$ from~\eqref{eqn:fixed_point_itn} in the present context); see~\eqref{eq:eta}.

To establish such a bound, we begin with a quasi-interpolation result.

\begin{lemma}\label{lm:int}
Consider a finite element mesh~$\mesh[h]$, and a corresponding FEM space~$\cgspace$ as in~\eqref{eq:FEMspace}. Moreover, let~$\pi:\, \Hp[0]{1}(\Omega)\to\cgspace$ be the
Cl\'ement interpolation operator~\cite{Clement:75}. Then,
\begin{equation*}
        \sum_{K\in\mesh}
            \left( \gamma_K^{-1} \norm*{v-\pi v}_{\Lp(K)}^2
                + \alpha_2\norm*{\nabla\left(v-\pi v\right)}_{\Lp(K)}^2
                +\frac12\alpha_2^{\nicefrac12}\gamma_K^{-\nicefrac12}\norm{v-\pi v}^2_{\Lp(\partial K)}
                \right)
            \leq C_I^2 \tnorm*{ v }^2
\end{equation*}
for all~$v\in\Hp[0]{1}(\Omega)$, with a constant~$C_I>0$ independent of the local element sizes, and
\begin{equation}\label{eq:gamma}
\gamma_K=\begin{cases}
\min\left(\alpha_2^{-1}h_K^{2},\beta_2^{-1}\right)&\text{if }\beta_2\neq 0,\\
h_K^{2}\alpha_2^{-1}&\text{otherwise},
\end{cases}
\end{equation}
for any~$K\in\mesh[h]$. Here, $\alpha_2$ and~$\beta_2$ are the constants from~\eqref{eqn:nonlinearity assump bound} and~\eqref{eqn:forcing-func assump bound}, respectively.
\end{lemma}

\begin{proof}
Let~$v\in\Hp[0]{1}(\Omega)$. We begin by recalling the following well-known approximation properties of the Cl\'ement interpolant:
\begin{align*}
h_K^{-2} \norm*{v-\pi v}_{\Lp(K)}^2
                + \norm*{\nabla\left(v-\pi v\right)}_{\Lp(K)}^2
            &\leq C \norm*{ \nabla v }_{\Lp(\omega_K)}^2,\\
            \norm*{v-\pi v}_{\Lp(K)}^2
                   &\leq C \norm*{ v }_{\Lp(\omega_K)}^2,
\end{align*}
for any~$K\in\mesh[h]$, with a constant~$C>0$ independent of the local element sizes and of~$v$; for~$K\in\mesh[h]$ we denote by~$\omega_K$ the patch of all elements in~$\mesh[h]$ adjacent to~$K$. In particular, following the approach in~\cite{Ve98}, this implies that
\begin{align*}
\norm*{v-\pi v}_{\Lp(K)}^2\le C\alpha_2^{-1}h_K^2\left(\alpha_2\norm{\nabla v}^2_{\Lp(\omega_K)}+\beta_2\norm*{ v }_{\Lp(\omega_K)}^2\right),
\end{align*}
and that, if~$\beta_2\neq 0$, then
\begin{align*}
\norm*{v-\pi v}_{\Lp(K)}^2\le C\beta_2^{-1} \left(\alpha_2\norm{\nabla v}^2_{\Lp(\omega_K)}+\beta_2\norm*{ v }_{\Lp(\omega_K)}^2\right).
\end{align*}
Therefore,
\begin{equation*}
\norm*{v-\pi v}_{\Lp(K)}^2\le C\gamma_K \left(\alpha_2\norm{\nabla v}^2_{\Lp(\omega_K)}+\beta_2\norm*{ v }_{\Lp(\omega_K)}^2\right),
\end{equation*}
and so
\begin{equation}\label{eq:aux10}
\gamma_K^{-1} \norm*{v-\pi v}_{\Lp(K)}^2
                + \alpha_2\norm*{\nabla\left(v-\pi v\right)}_{\Lp(K)}^2
            \leq C\left(\alpha_2\norm{\nabla v}^2_{\Lp(\omega_K)}+\beta_2\norm*{ v }_{\Lp(\omega_K)}^2\right).
\end{equation}

Moreover, using the multiplicative trace inequality, that is,
\[
\norm{v-\pi v}^2_{\Lp(\partial K)}\le C\left(h_K^{-1}\norm{v-\pi v}_{\Lp(K)}^2+\norm{v-\pi v}_{\Lp(K)}\norm{\nabla v-\pi v}_{\Lp(K)}\right)\qquad \forall K\in\mesh,
\]
we infer that
\begin{align*}
\norm{v-\pi v}^2_{\Lp(\partial K)}
&\le C\left(h_K^{-1}\gamma_K+\gamma_K^{\nicefrac12}\alpha_2^{-\nicefrac12}\right)\left(\alpha_2\norm{\nabla v}^2_{\Lp(\omega_K)}+\beta_2\norm*{ v }_{\Lp(\omega_K)}^2\right).
\end{align*}
Observing that
\[
h_K^{-1}\gamma_K+\gamma_K^{\nicefrac12}\alpha_2^{-\nicefrac12}
= \left(h_K^{-1}\alpha_2^{\nicefrac12}\gamma_K^{\nicefrac12}+1\right)\gamma_K^{\nicefrac12}\alpha_2^{-\nicefrac12}\le 2\gamma_K^{\nicefrac12}\alpha_2^{-\nicefrac12},
\]
we now arrive at
\begin{equation}\label{eq:aux11}
\alpha_2^{\nicefrac12}\gamma_K^{-\nicefrac12}\norm{v-\pi v}^2_{\Lp(\partial K)}\le C\left(\alpha_2\norm{\nabla v}^2_{\Lp(\omega_K)}+\beta_2\norm*{ v }_{\Lp(\omega_K)}^2\right).
\end{equation}

Finally, combining~\eqref{eq:aux10} and~\eqref{eq:aux11}, and summation over all~$K\in\mesh[h]$ concludes the argument.
\end{proof}

In order to formulate the following result, we consider a series of meshes, $\{\mesh[h,i]\}_{i\geq 0}$; for each mesh $\mesh[h,i]$ we denote the finite element space on that mesh as $\cgspacei$.

\begin{theorem}
    \label{thm:a posteriori}
    Let $u \in \Hp[0]{1}(\Omega)$ be the exact solution to the boundary value problem
    \eqref{eqn:nonlinear PDE}--\eqref{eqn:nonlinear PDE BC}, and $\mesh[h,0]$ be an initial mesh with initial guess
    $u_{h,0}^{0}\in\cgspacei[0]$. Moreover, denote by~$\mesh[h,i]$ the mesh after $i$~mesh refinements, and let
    $u_{h,i}^{n}\in\cgspacei$ be the FEM solution obtained after $n$ steps of the fixed point iteration \eqref{eqn:fixed_point_itn} 
    on~$\mesh[h,i]$. Here, the initial guess $u_{h,i}^{0}\in\cgspacei$ on the current mesh~$\mesh[h,i]$, $i> 0$, is obtained
    as an (appropriate) projection  of the solution~$u_{h,i-1}^{n^\star}\in\cgspacei[i-1]$ of the last ($n^\star$-th) iteration on the mesh~$\mesh[h,i-1]$
    to the space $\cgspacei$. Then, for $n\geq 1$, there holds the {\em a posteriori} error estimate
    \begin{equation*}
        \tnorm*{u-u_{h,i}^{n}} \leq C_I \Bigg(\sum_{K\in\mesh[h,i]} \eta_K^2\Bigg)^{\nicefrac{1}{2}}
            + L\left(1+L\right) \tnorm*{u_{h,i}^{n}-u_{h,i}^{n-1}},
    \end{equation*}
    where $C_I$ is the constant from Lemma~\ref{lm:int}, and
    \begin{align*}
    \eta_K^2 &= \gamma_K \norm*{f\left(u_{h,i}^{n-1}\right) -\nabla\cdot\left(\mu\left(\abs*{\nabla u_{h,i}^{n-1}}\right)\nabla u_{h,i}^{n-1}\right) +L^2\F\left( u_{h,i}^{n}-u_{h,i}^{n-1}\right) }_{\Lp(K)}^2 \\
            &\quad+ \frac12\alpha_2^{-\nicefrac12}{\gamma_K}^{\nicefrac12}\norm*{ \jmp{ \mu\left(\abs*{\nabla u_{h,i}^{n-1}}\right)\nabla u_{h,i}^{n-1} + L^2 \alpha_2\nabla\left(u_{h,i}^{n}-u_{h,i}^{n-1}\right) } }_{\Lp(\partial K\setminus\Gamma)}^2,
    \end{align*}
    for any $K\in\mesh[h,i]$ and~$n\ge 1$, are \emph{local error indicators}. Here, $\gamma_K$, $K\in\mesh[h,i]$, is defined in~\eqref{eq:gamma} and 
    \[
    \F(v) = -\alpha_2\Delta v + \beta_2 v.
    \]
    Moreover, for an edge $ e\subset\partial K^+ \cap \partial K^{-}$ between two neighbouring elements~$K^\pm\in\mesh[h,i]$, we signify by $\jmp{\bm v} \!\big\vert_e = \bm v^{+}\vert_e\cdot\vec{n}_{K^{+}} + \bm v^{-}\vert_e\cdot \vec{n}_{K^{-}}$ the jump of a (vector-valued) function~$\bm v$ along~$e$, where $\bm v^\pm$ denote the traces of the function $\bm v$ on the edge $e$ taken from the
    interior of $K^\pm$, respectively, and $\vec{n}_{K^{\pm}}$ are the unit outward normal vectors on $\partial K^\pm$, respectively.
\end{theorem}

\begin{proof}
Recalling our abstract result, Proposition~\ref{pr:aposteriori}, it is sufficient to derive a quantity~$\eta(u_{h,i}^n,\cgspacei)$ such that
\[
\tnorm{\widetilde{u}_i^n-u_{h,i}^n}\le\eta(u_{h,i}^n,\cgspacei);
\]
cf.~\eqref{eq:eta}. The reconstruction~$\widetilde{u}_i^n\in\Hp[0]{1}(\Omega)$ fulfills
\[
(\widetilde{u}_i^n,v)_\Omega=(u_{h,i}^{n-1},v)_\Omega-L^{-2}A(u_{h,i}^{n-1},v)\qquad\forall v\in\Hp[0]{1}(\Omega)\supset\cgspacei,
\]
see~\eqref{eq:auxiliary}, where~$u_{h,i}^n\in\cgspacei$ is just the Galerkin approximation of $\widetilde{u}_i^n$ (cp.~\eqref{eqn:fixed_point_itn}).

Define the error~$\wt e_{h,i}^{n} = \wt u_i^n-u_{h,i}^{n}$, and let~$v_{h,i}=\pi \wt e_{h,i}^{n} \in\cgspacei$, where~$\pi$ is the interpolation operator from Lemma~\ref{lm:int}. We notice that there holds
\begin{align*}
        \tnorm{\wt e^n_{h,i}}^2 
        &=(\wt u_i^n,\wt e^n_{h,i})_\Omega-(u^n_{h,i},\wt e^n_{h,i})_\Omega\\
        &=(\wt u_i^n,\wt e^n_{h,i}-v_{h,i})_\Omega-(u^n_{h,i},\wt e^n_{h,i}-v_{h,i})_\Omega\\
        &= - L^{-2}A(u_{h,i}^{n-1}, \wt e_{h,i}^{n} - v_{h,i}) - \left( u_{h,i}^{n} - u_{h,i}^{n-1}, \wt e_{h,i}^{n} - v_{h,i} \right)_\Omega.
\end{align*}
Integration by parts elementwise leads to
    \begin{align*}
        L^2\tnorm{\wt e^n_{h,i}}^2
        &=-\sum_{K\in\mesh[h,i]}\int_K \left(\mu\left(\abs*{\nabla(u_{h,i}^{n-1})}\right)\nabla u_{h,i}^{n-1} + L^2\alpha_2\nabla\left(u_{h,i}^{n}-u_{h,i}^{n-1}\right)\right)\cdot\nabla\left(\wt e_{h,i}^{n}-v_{h,i}\right)\dx \\
            &\quad -\sum_{K\in\mesh[h,i]}\int_K \left(f\left(u_{h,i}^{n-1}\right) +L^2\beta_2\left(u_{h,i}^{n}-u_{h,i}^{n-1}\right)\right)\left(\wt e_{h,i}^{n}-v_{h,i}\right)\dx\\
        &= \sum_{K\in\mesh[h,i]}\int_K \nabla\cdot\left(\mu\left(\abs*{\nabla u_{h,i}^{n-1}}\right)\nabla u_{h,i}^{n-1} + L^2\alpha_2 \nabla\left(u_{h,i}^{n}-u_{h,i}^{n-1}\right) \right) \left(\wt e_{h,i}^{n}-v_{h,i}\right) \dx \\
            &\quad -\sum_{K\in\mesh[h,i]}\int_K \left(f\left(u_{h,i}^{n-1}\right) +L^2\beta_2\left(u_{h,i}^{n}-u_{h,i}^{n-1}\right)\right)\left(\wt e_{h,i}^{n}-v_{h,i}\right)\dx\\
            & \quad -\sum_{K\in\mesh[h,i]} \int_{\partial K\setminus\Gamma} \left(\mu\left(\abs*{\nabla u_{h,i}^{n-1}}\right)\nabla u_{h,i}^{n-1} + L^2 \alpha_2 \nabla\left( u_{h,i}^{n}-u_{h,i}^{n-1} \right) \right) \cdot \vec{n}_K \left(\wt e_{h,i}^{n}-v_{h,i}\right) \ds.
            \end{align*}
A few elementary calculations show that
\begin{align*}
\sum_{K\in\mesh[h,i]}& \int_{\partial K\setminus\Gamma} \left(\mu\left(\abs*{\nabla u_{h,i}^{n-1}}\right)\nabla u_{h,i}^{n-1} + L^2 \alpha_2 \nabla\left( u_{h,i}^{n}-u_{h,i}^{n-1} \right) \right)\cdot \vec{n}_K \left(\wt e_{h,i}^{n}-v_{h,i}\right) \ds\\
&=\frac12\sum_{K\in\mesh[h,i]} \int_{\partial K\setminus\Gamma} \jmp{\mu\left(\abs*{\nabla u_{h,i}^{n-1}}\right)\nabla u_{h,i}^{n-1} + L^2 \alpha_2 \nabla\left( u_{h,i}^{n}-u_{h,i}^{n-1} \right)} \left(\wt e_{h,i}^{n}-v_{h,i}\right) \ds,
\end{align*}
and thus, using the Cauchy-Schwarz inequality, implies
    \begin{align*}
      L^2\tnorm{\wt e^n_{h,i}}^2
       &\leq \sum_{K\in\mesh[h,i]} \gamma^{\nicefrac12}_K \norm*{f\left(u_{h,i}^{n-1}\right) -\nabla\cdot\left(\mu\left(\abs*{\nabla u_{h,i}^{n-1}}\right)\nabla u_{h,i}^{n-1}\right) +L^2\F\left( u_{h,i}^{n}-u_{h,i}^{n-1}\right) }_{\Lp(K)} \\
            &\qquad\qquad \times\gamma_K^{-\nicefrac12} \norm*{\wt e_{h,i}^{n}-v_{h,i}}_{\Lp(K)} \\
            & \quad + \frac{1}{2}\sum_{K\in\mesh[h,i]}  \alpha_2^{-\nicefrac14}\gamma_K^{\nicefrac{1}{4}}\norm*{\jmp{ \mu\left(\abs*{\nabla u_{h,i}^{n-1}}\right)\nabla u_{h,i}^{n-1} + L^2 \alpha_2\nabla \left( u_{h,i}^{n}-u_{h,i}^{n-1} \right) }}_{\Lp(\partial K\setminus\Gamma)}\\&\qquad\qquad \times\alpha_2^{\nicefrac14}\gamma_K^{-\nicefrac{1}{4}} \norm*{\wt e_{h,i}^{n}-v_{h,i}}_{\Lp(\partial K \setminus \Gamma)} \\
            &\leq \Bigg( \sum_{K\in\mesh[h,i]} \eta_K^2 \Bigg)^{\nicefrac{1}{2}} \Bigg( \sum_{K\in\mesh[h,i]} \left( \gamma_K^{-1} \norm*{\wt e_{h,i}^{n}-v_{h,i}}_{\Lp(K)}^2 + \frac{\alpha_2^{\nicefrac12}}{2\gamma_K^{\nicefrac12}} \norm*{\wt e_{h,i}^{n}-v_{h,i}}_{\Lp(\partial K \setminus \Gamma)}^2 \right) \Bigg)^{\nicefrac{1}{2}}.
    \end{align*}
Therefore, we infer that, employing Lemma~\ref{lm:int},
  \begin{equation*}
         L^2\tnorm{\wt e^n_{h,i}}^2\leq C_I \Bigg( \sum_{K\in\mesh[h,i]} \eta_K^2 \Bigg)^{\nicefrac{1}{2}} \tnorm*{\wt e_{h,i}^{n}},
    \end{equation*}
    which implies
    \[
    \tnorm{\wt e^n_{h,i}}\leq C_IL^{-2} \Bigg( \sum_{K\in\mesh[h,i]} \eta_K^2 \Bigg)^{\nicefrac{1}{2}}=:\eta(u^n_{h,i},\cgspacei).
    \]
Inserting this bound into~\eqref{eq:aposteriori} with~$c_0=1$ (cp.~Proposition~\ref{pr:monotone}) completes the proof.
\end{proof}
\subsection{Adaptive refinement algorithm}
Proceeding along the lines of Section~\ref{sc:adaptive}, we notice that the \emph{a posteriori} error bound from \thref{thm:a posteriori} controls the error in terms of two contributions: The \emph{finite element error}, defined as
\[
    \mathcal{E}_{{\sf FEM},i}^{n} = \Bigg( \sum_{K\in\mesh[h,i]} \eta_K^2 \Bigg)^{\nicefrac{1}{2}},
\]
and the \emph{fixed point error}
\[
    \mathcal{E}_{{\sf FP},i}^{n} = L(1+L) \tnorm*{u_{h,i}^{n}-u_{h,i}^{n-1}}.
\]
This allows us to write the error bound as
\[
    \tnorm*{u-u_{h,i}^{n}} \leq C_I \mathcal{E}_{{\sf FEM},i}^{n} + \mathcal{E}_{{\sf FP},i}^{n}.
\]
Based on this bound we can cast the abstract adaptive Algorithm~\ref{algo:refinement_abstract} into the fixed point Galerkin iteration~\eqref{eqn:fixed_point_itn} for the solution of~\eqref{eqn:nonlinear PDE}--\eqref{eqn:nonlinear PDE BC}.

\begin{algorithm}
    \label{algo:refinement}
    Choose an initial starting mesh $\mesh[h,0]$, and an initial guess
    $u_{h,0}^{0}\in\cgspacei[0]$ in the associated finite element space $\cgspacei[0]$ (of fixed polynomial degree~$p\ge 1$).
    \begin{algorithmic}
        	\For{$i \gets 0,1,2,\dots$}
	        \State $n \gets 0$
	        \Repeat
	            \State $n \gets n+1$
	            \State Perform a single fixed point iteration \eqref{eqn:fixed_point_itn} to calculate $u_{h,i}^{n}$.
	        \Until{$\mathcal{E}_{{\sf FP},i}^{n} \leq \vartheta \mathcal{E}_{{\sf FEM},i}^{n}$}
	        \State Perform mesh refinement (and/or derefinement) on $\mesh[h,i]$ 
	        \State based on the error indicators~$\eta_K$ from Theorem~\ref{thm:a posteriori} 
	        \State together with a suitable marking strategy in order to obtain $\mesh[h,i+1]$.
	        \State $u_{h,i+1}^{0} \gets \pi_{i,i+1} u_{h,i}^{n}$
	    \EndFor
	\end{algorithmic}
	Here, $\pi_{i,i+1}$ is some projection from $\cgspacei$ to $\cgspacei[i+1]$ (for instance, the $(.,.)_\Omega$-projection),
	and $\vartheta>0$ is a (prescribed) parameter.
\end{algorithm}

\subsection{Numerical experiments}
\label{sec:numerics}
In this section we perform a series of numerical experiments to validate the \emph{a priori} and \emph{a posteriori}
error bounds for the fixed point iteration~\eqref{eqn:fixed_point_itn} from \thref{theorem:a priori fp} and \thref{thm:a posteriori}, respectively.

    \subsubsection{Validation of Remark~\ref{rm:h-p-FEM}}
    We consider the domain $\Omega= (0,1)^2 \subset \real^2$ with nonlinearity
    \[
        \mu(\abs{\nabla u}) = 2 + \frac{1}{1 + \abs{\nabla u}^2},
    \]
    and select $f$ independent of $u$ such that the analytical solution to
    \eqref{eqn:nonlinear PDE}--\eqref{eqn:nonlinear PDE BC} is given by
    \[
        u(x, y) = x(1-x)y(1-y)(1-2y)\e^{-20(2x-1)^2}.
    \]
    We note that $\beta_1=\beta_2=0$, $\alpha_1=3$ and $\alpha_2=\nicefrac{15}{8}$.

    \begin{figure}[t]
        \centering
        \subfloat[]{\label{fig:priori:p}\includegraphics[width=0.49\textwidth]{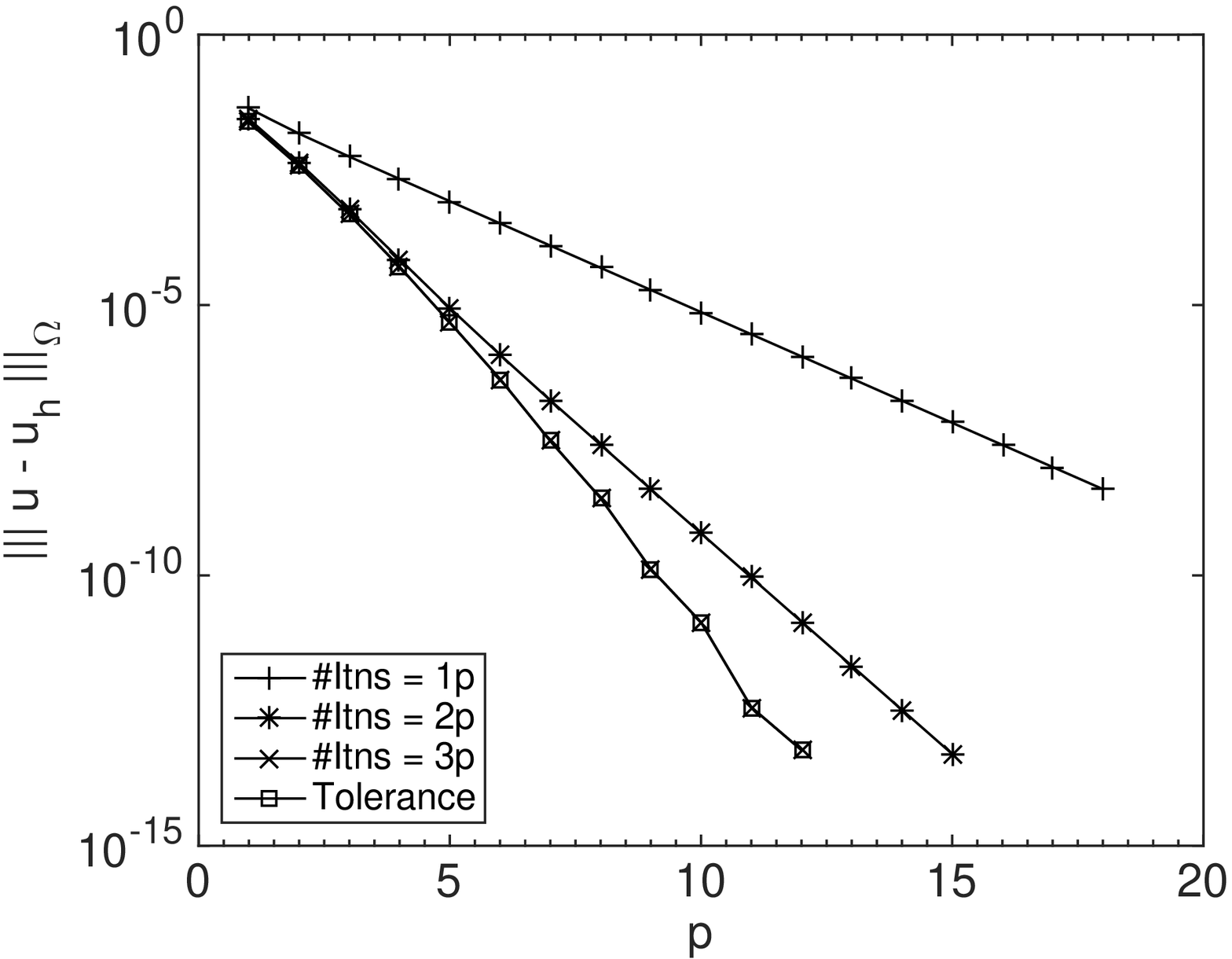}}
        \subfloat[]{\label{fig:priori:h}\includegraphics[width=0.49\textwidth]{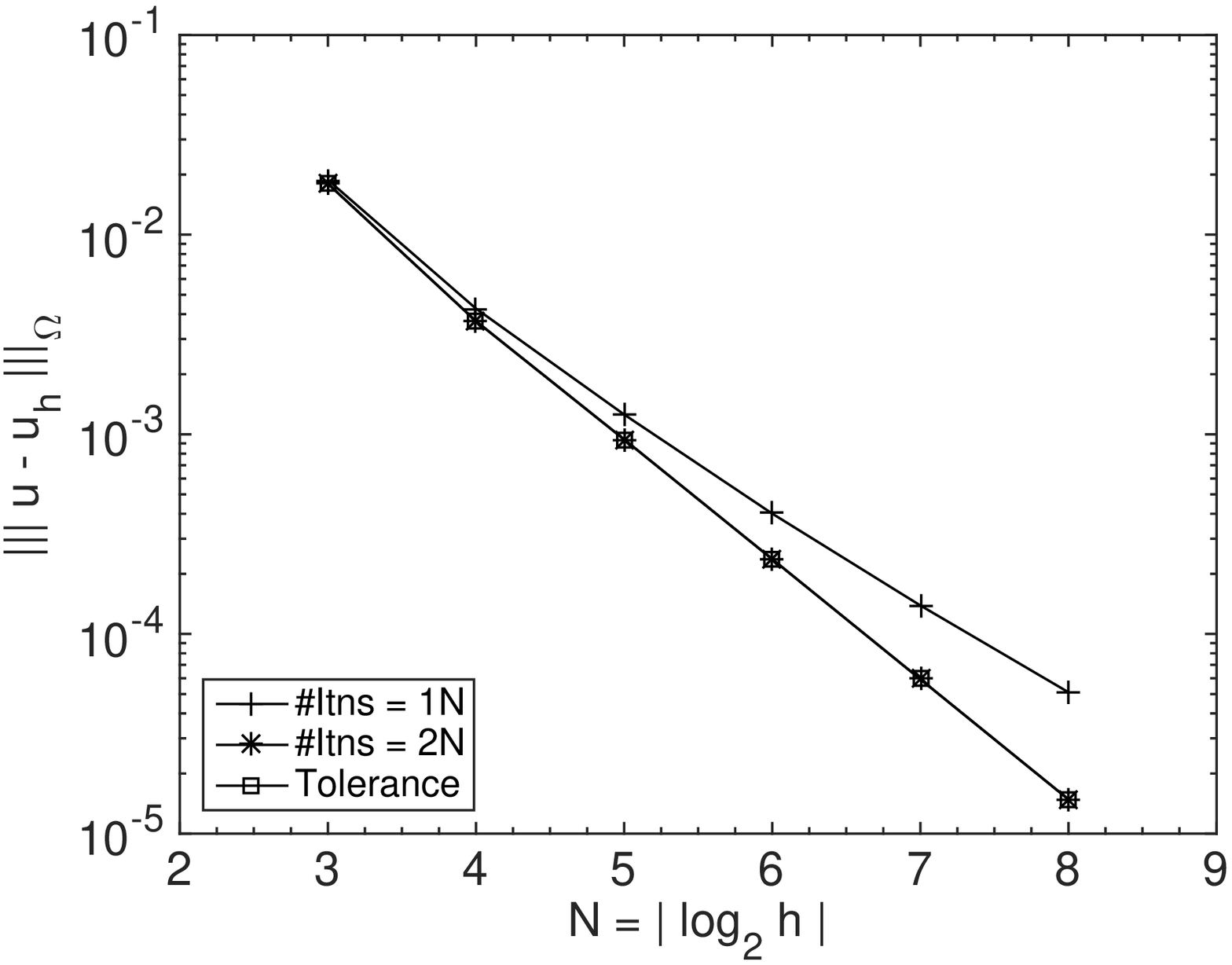}}
        \caption{Convergence with restricted number of fixed point iterations
            with uniform \protect\subref{fig:priori:p} $p$-refinement and \protect\subref{fig:priori:h} $h$-refinement.}
        \label{fig:priori}
    \end{figure}

    Firstly, we consider the case when the mesh is fixed as a $16 \times 16$ uniform square mesh of quadrilaterals
    and perform uniform refinement of the polynomial degree $p$ from an initial guess~$u^0_{h,0}\equiv 0$. In this situation
    we restrict the number of iterations of the fixed point iteration to $C_p\cdot p$, for $C_p=1,2,3$ and plot in
    \subfigref{fig:priori}{p} the
    error $\tnorm{u-u_h}$ against the polynomial degree $p$. For comparison, we also perform the same experiment
    continuing the fixed point iteration until the residual $A(u_h^{n},v)$ is below a given tolerance~($10^{-14}$) and,
    hence, the approximation is close to the best possible FEM approximation for the mesh. We clearly observe that by restricting
    the number of iterations we obtain exponential convergence of the error, and when $C_p=3$ we gain the same
    convergence rate as allowing the iteration to continue until a tolerance is reached. Hence only performing
    the iteration $3p$ times gives an optimal convergence rate in the given example.

    Secondly, we consider a fixed polynomial degree $p=2$ and perform $h$ refinement to generate a sequence of $2^N \times 2^N$ uniform
    square meshes of quadrilaterals, for $N=3,\dots,8$. We again perform both a restriction of the number of iterations
    of the fixed point method to $C_N\cdot N$, for $C_N=1,2$, as well as allowing the iteration to continue until a tolerance is
    reached, and plot in \subfigref{fig:priori}{h} the error $\tnorm{u-u_h}$ against $|\log_2 h|$. We obtain algebraic convergence, and already when $C_N=2$ we achieve the optimal convergence rate~$\mathcal{O}(h^2)$.
    
    \subsubsection{Validation of \thref{thm:a posteriori} and Remark~\ref{rm:epsilon}}
    We now consider automatic $h$-adaptive mesh refinement, with linear ($p=1$) basis functions, using Algorithm~\ref{algo:refinement} and the \emph{a posteriori} error bound
    from \thref{thm:a posteriori}. For the purpose of mesh refinement we use a fixed fraction refinement strategy, where the $25\%$ of elements
    with the largest local error indicators $\eta_K$ are marked for refinement, and the $5\%$ of elements with the smallest
    local error indicators are marked for derefinement.\medskip
    
        \emph{Example 1. Nonlinear diffusion}
        We first consider the case of a $u$-independent $f$ with a nonlinear $\mu$ on the unit square $\Omega=(0,1)^2 \subset \real^2$.
        To this end, we let
        \[
            \mu(\abs{\nabla u}) = 1 + \arctan(\abs{\nabla u}^2),
        \]
        and select $f$ such that the analytical solution to
        \eqref{eqn:nonlinear PDE}--\eqref{eqn:nonlinear PDE BC} is given by
        \[
            u(x, y) = x(1-x)y(1-y)(1-2y)\e^{-20(2x-1)^2}.
        \]
        We note that $\beta_1=\beta_2=0$, $\alpha_1=1 + \nicefrac{\sqrt{3}}{2}+\nicefrac{\pi}{3}$ and $\alpha_2=1$. For this problem we set
        the steering parameter $\vartheta =\nicefrac{1}{2}$ in Algorithm~\ref{algo:refinement}.
        
        \begin{figure}[t]
            \centering
            \subfloat[]{\label{fig:adaptive:square:err}\includegraphics[width=0.49\textwidth]{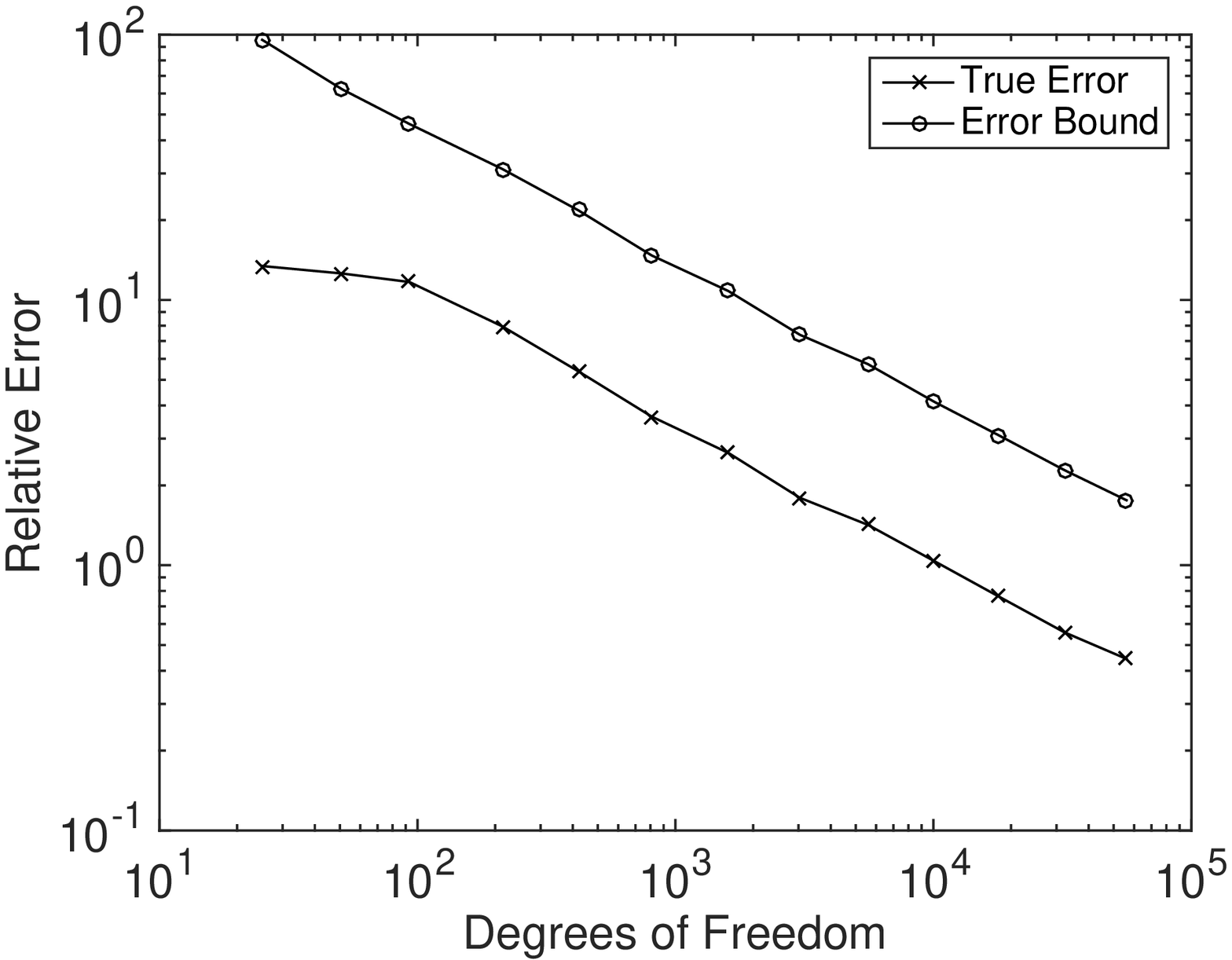}}
            \subfloat[]{\label{fig:adaptive:square:eff}\includegraphics[width=0.49\textwidth]{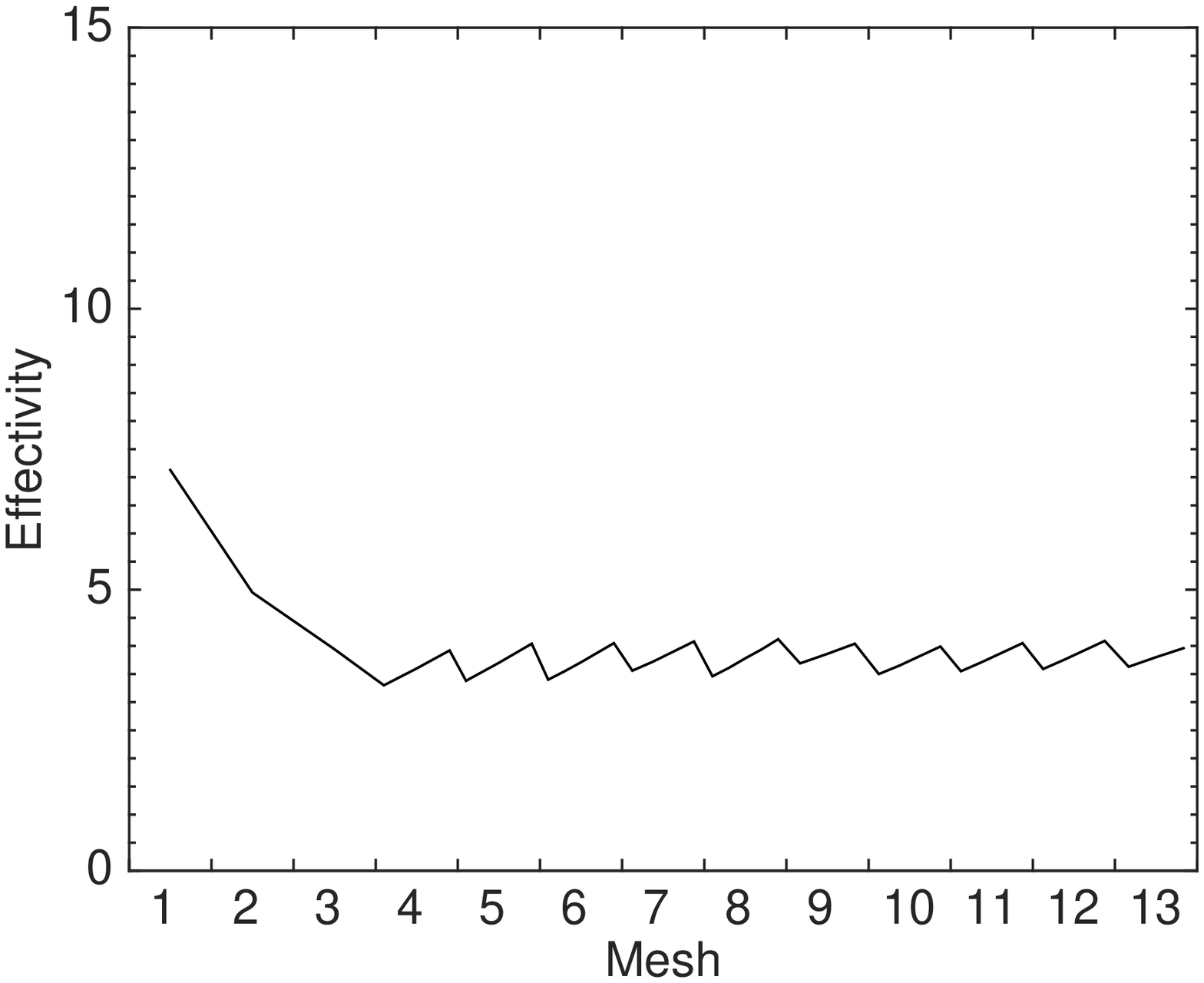}} \\
            \subfloat[]{\label{fig:adaptive:square:itn}\includegraphics[width=0.49\textwidth]{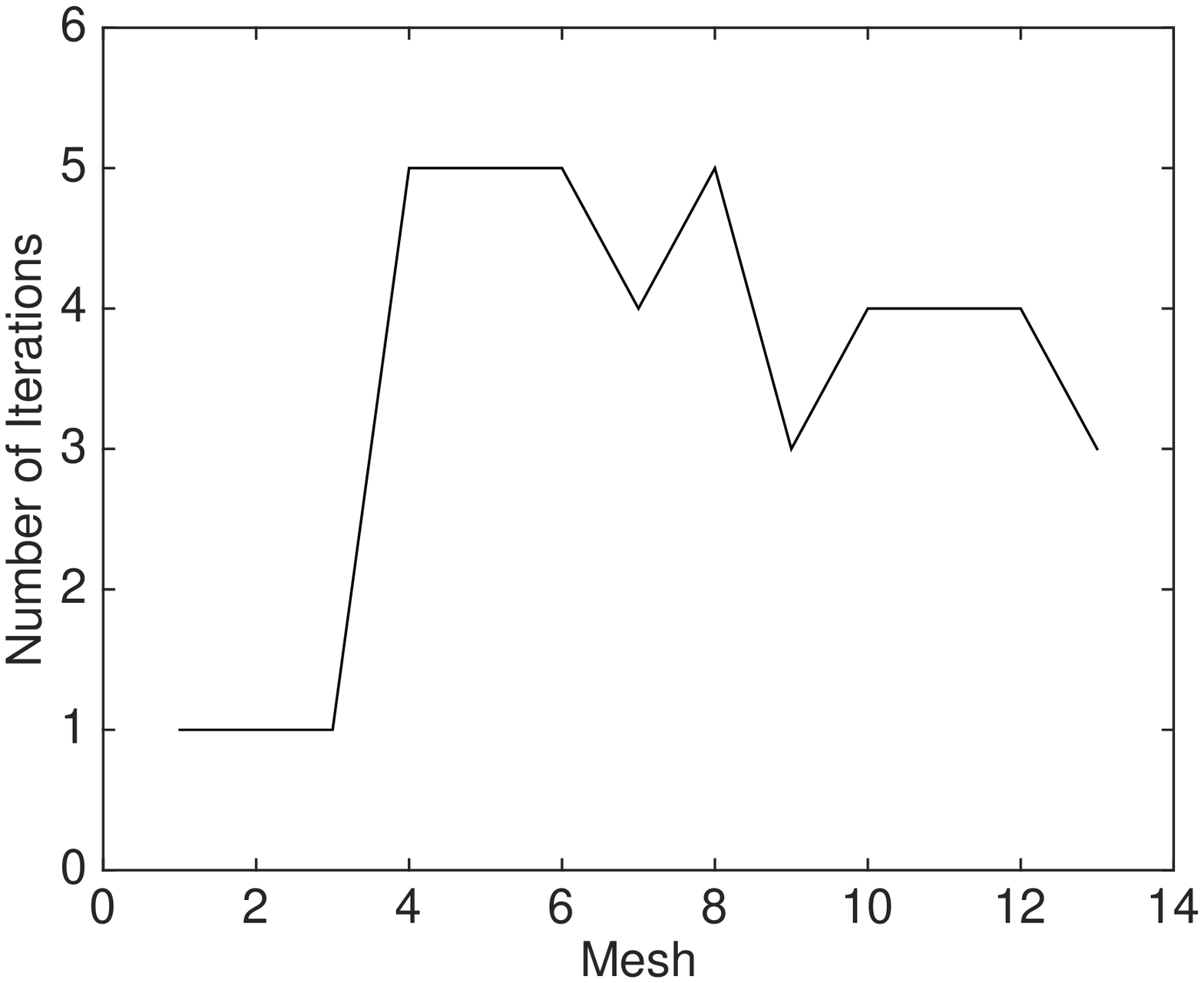}}
            \subfloat[]{\label{fig:adaptive:square:mesh}\includegraphics[width=0.49\textwidth]{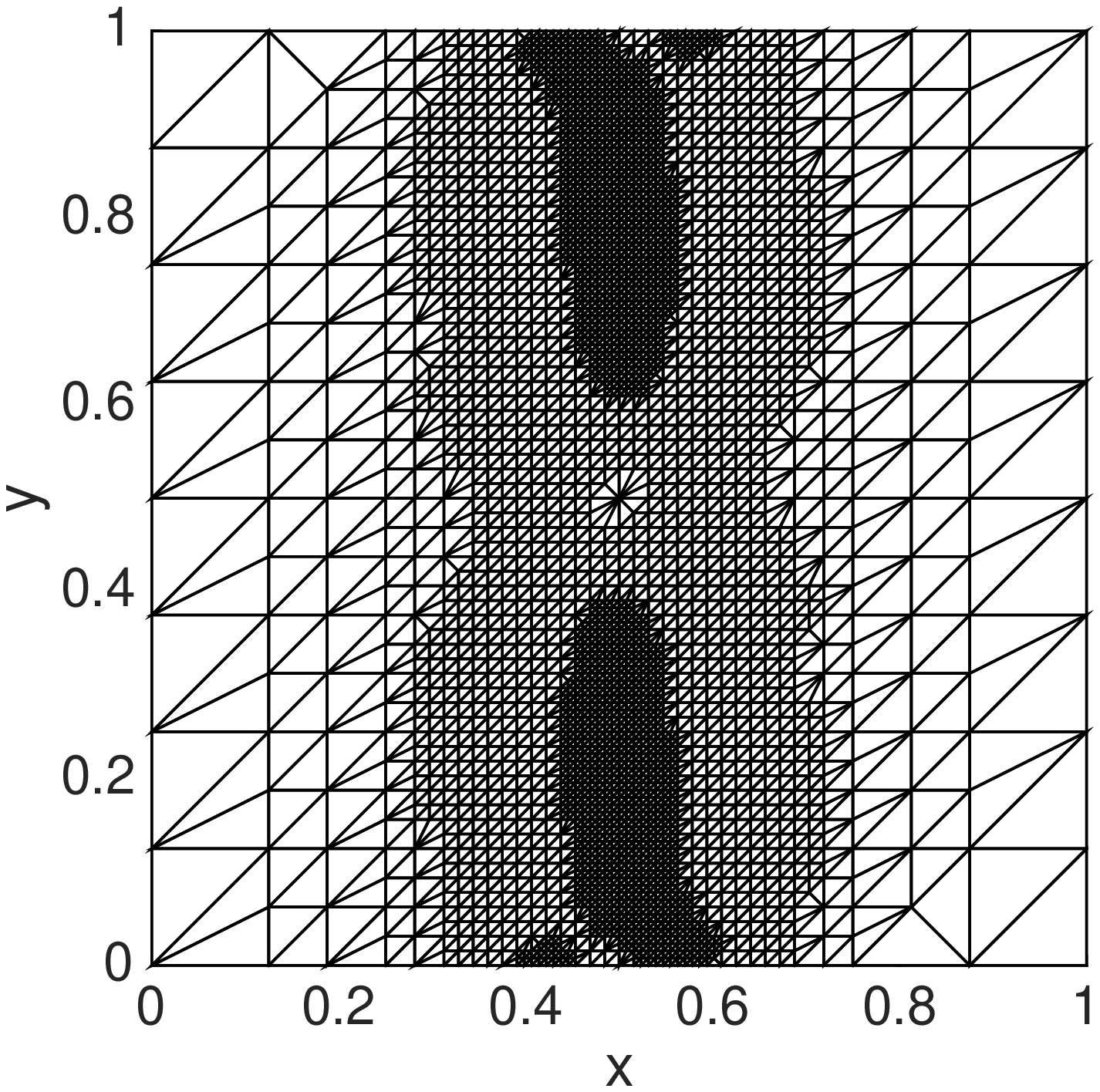}}
            \caption{Example 1. \protect\subref{fig:adaptive:square:err} Error in $\tnorm{\cdot}$-norm and error bound from
                \thref{thm:a posteriori} after the final fixed point iteration on each mesh compared to the number of degrees of freedom;
                \protect\subref{fig:adaptive:square:eff} Effectivity at each fixed point iteration for all meshes;
                \protect\subref{fig:adaptive:square:itn} Number of fixed point iterations on each mesh;
                \protect\subref{fig:adaptive:square:mesh} Mesh $\mesh[h,7]$ after 7 $h$-refinements.}
            \label{fig:adaptive:square}
        \end{figure}
        
        We first plot, in \subfigref{fig:adaptive:square}{err}, the relative true error $\nicefrac{\tnorm*{u-u_{h,i}^{n^\star}}}{\tnorm{u}}$ and the
        error bound, with $C_I=1$, from \thref{thm:a posteriori} after the last iteration $n^\star$ on each mesh $i$ against the number of
        degrees of freedom on that mesh. As can be seen, both the true error and the error bound converge at the same
        rate, and the error bound appears to overestimate the true error by a roughly constant amount. We also consider in
        \subfigref{fig:adaptive:square}{eff} the effectivity index at each step of the fixed point iteration on each mesh, where
        the effectivity index is the error bound (calculated with $C_I=1$) divided by the true error. As can be seen this is
        roughly constant (approximately $4$) for all meshes and iterations, which indicates that the error bound overestimates the
        true error by roughly this amount, independent of mesh properties. We do note, however, that the effectivity rises slightly
        due to the fixed point iteration on each mesh, this is likely caused by setting $C_I=1$. We also plot in
        \subfigref{fig:adaptive:square}{itn} the number of fixed point iterations at each mesh step required to ensure that
        the fixed point error is less than the finite element error. We note this is fairly constant although minor variations exist.
        A few mesh refinements are made early on which is likely caused by the fact that the 
        features of the solution are not accurately captured at the beginning.   
                
        In \subfigref{fig:adaptive:square}{mesh} we plot the mesh $\mesh[h,7]$ after 7 $h$-adaptive mesh refinements.
        The areas of mesh refinement coincide with the hill and valley in the analytical solution, which is the location we would
        expect the greatest error to occur, and matches the sort of refinement that occurs when the nonlinear methods are computed
        to a minimal residual. This suggests that the mesh refinement algorithm behaves in the expected manner.\medskip
        
        \begin{figure}[t]
            \centering
            \subfloat[]{\label{fig:adaptive:reac_diff:err}\includegraphics[width=0.49\textwidth]{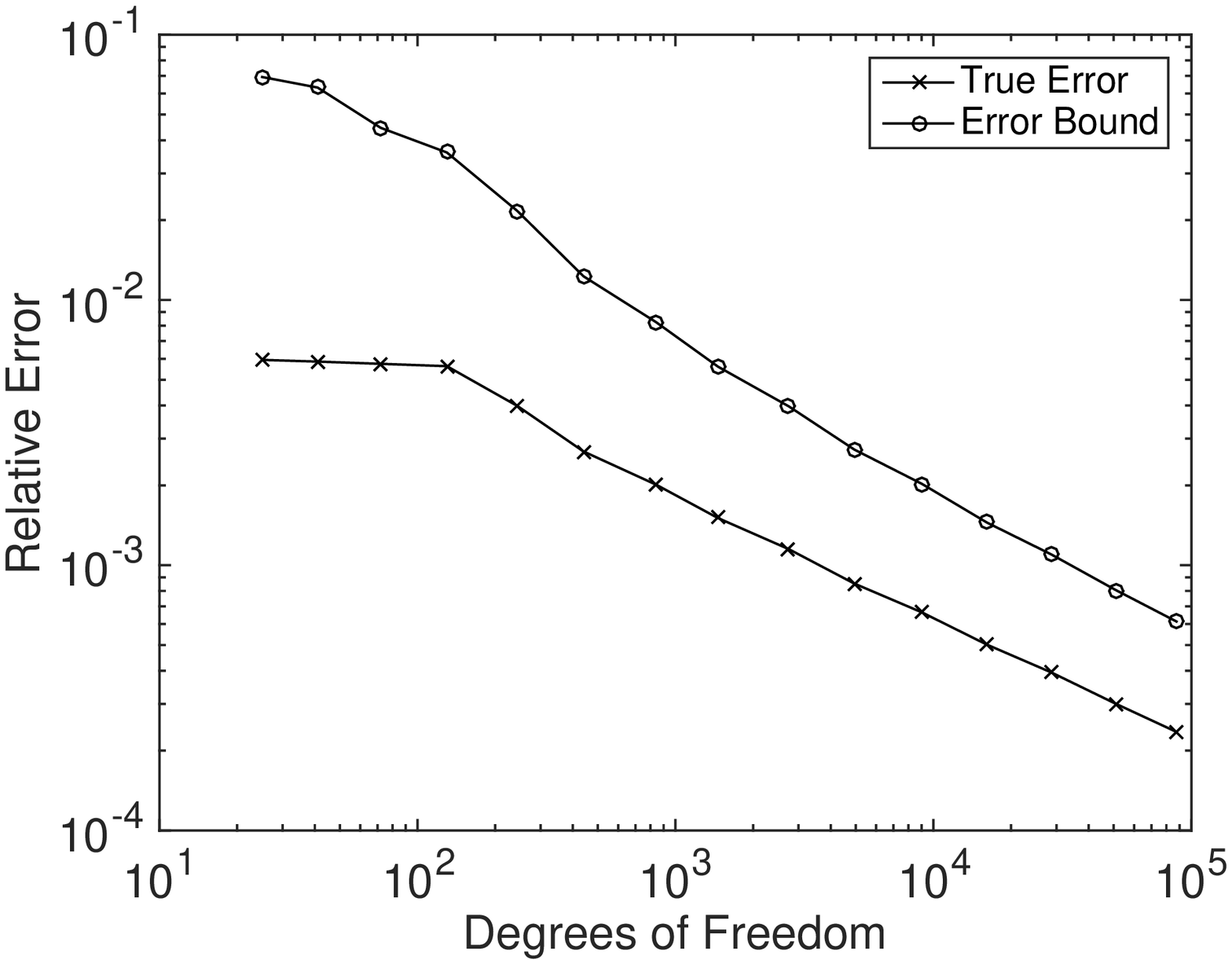}}
            \subfloat[]{\label{fig:adaptive:reac_diff:eff}\includegraphics[width=0.49\textwidth]{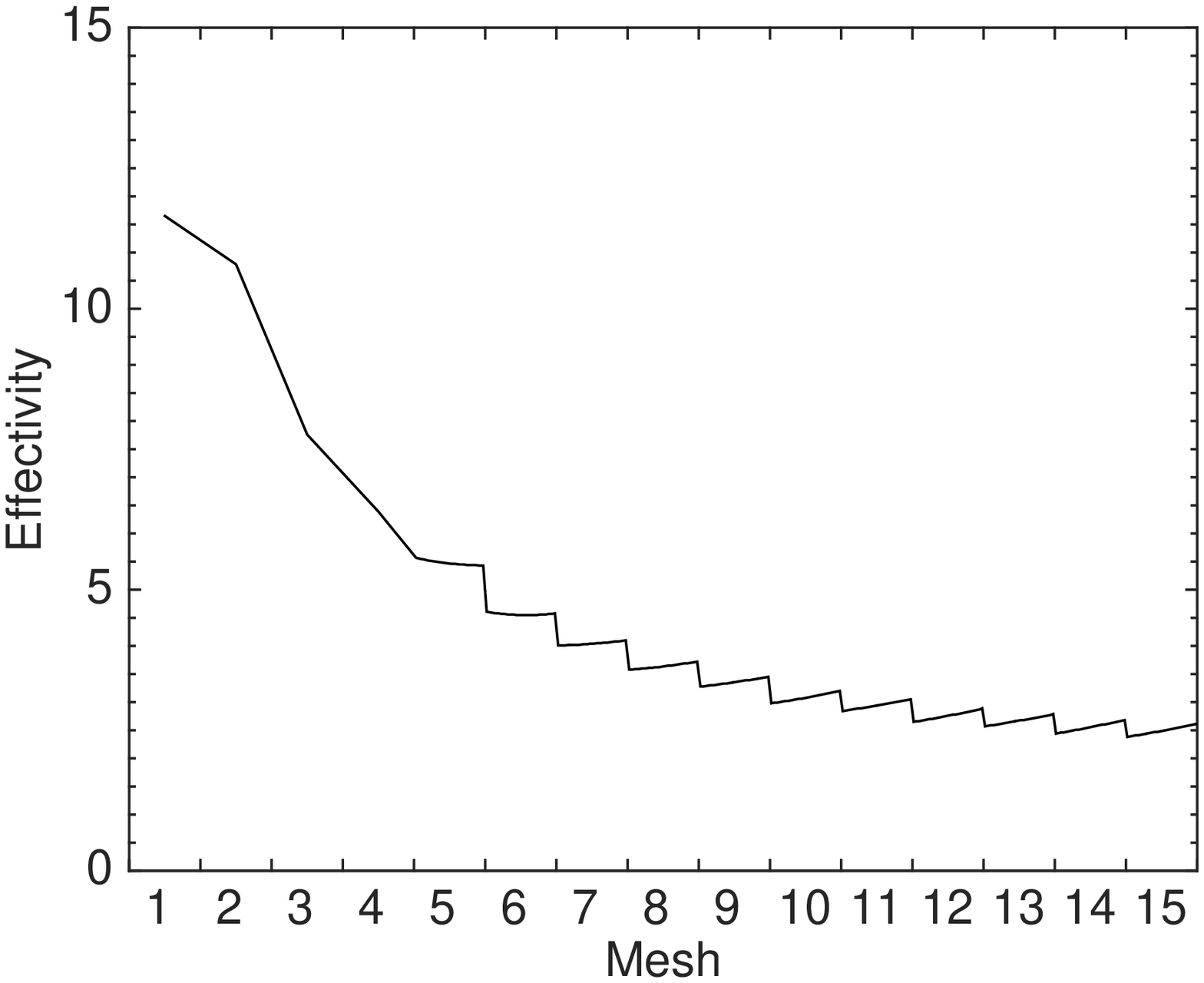}} \\
            \subfloat[]{\label{fig:adaptive:reac_diff:itn}\includegraphics[width=0.49\textwidth]{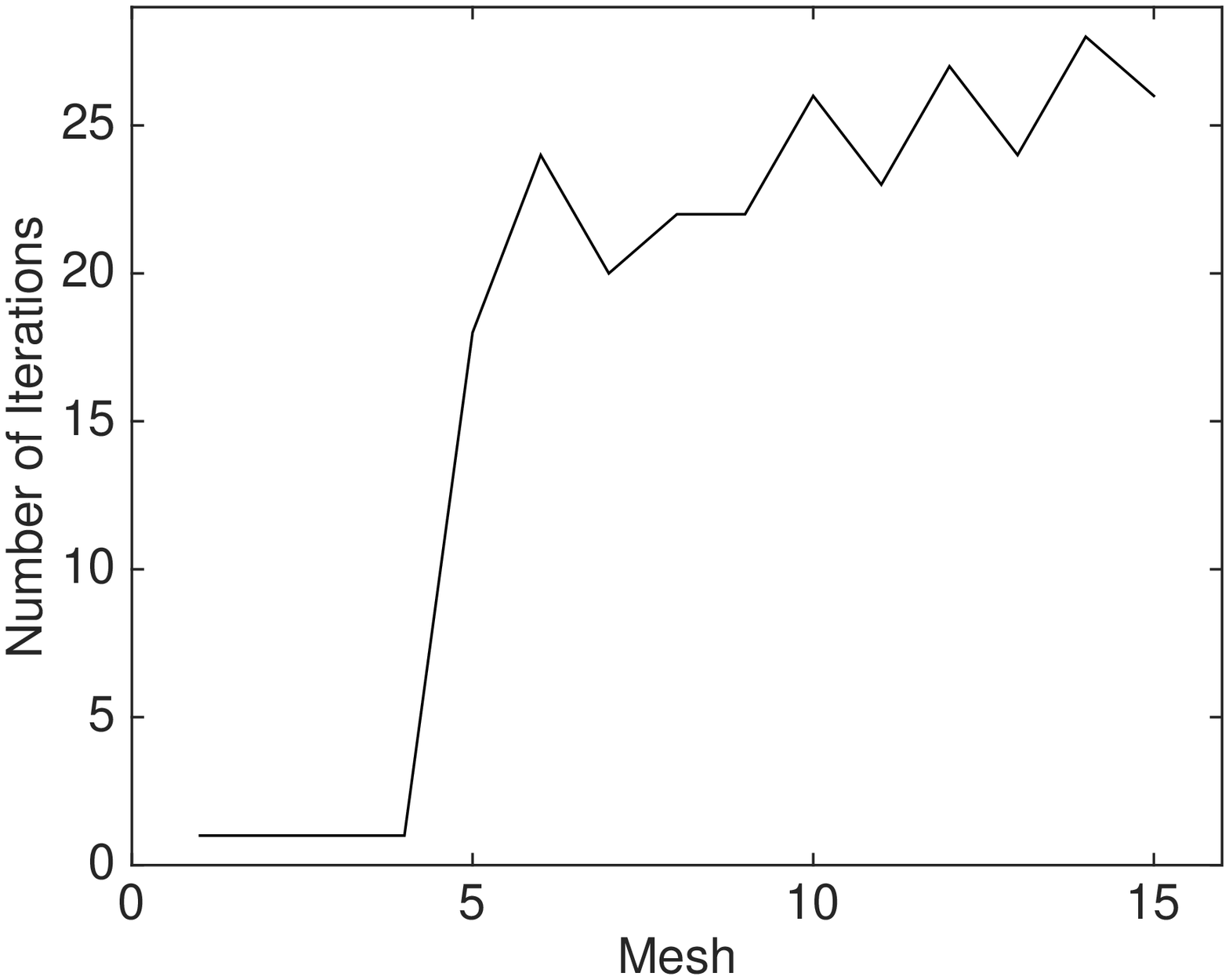}}
            \subfloat[]{\label{fig:adaptive:reac_diff:mesh}\includegraphics[width=0.49\textwidth]{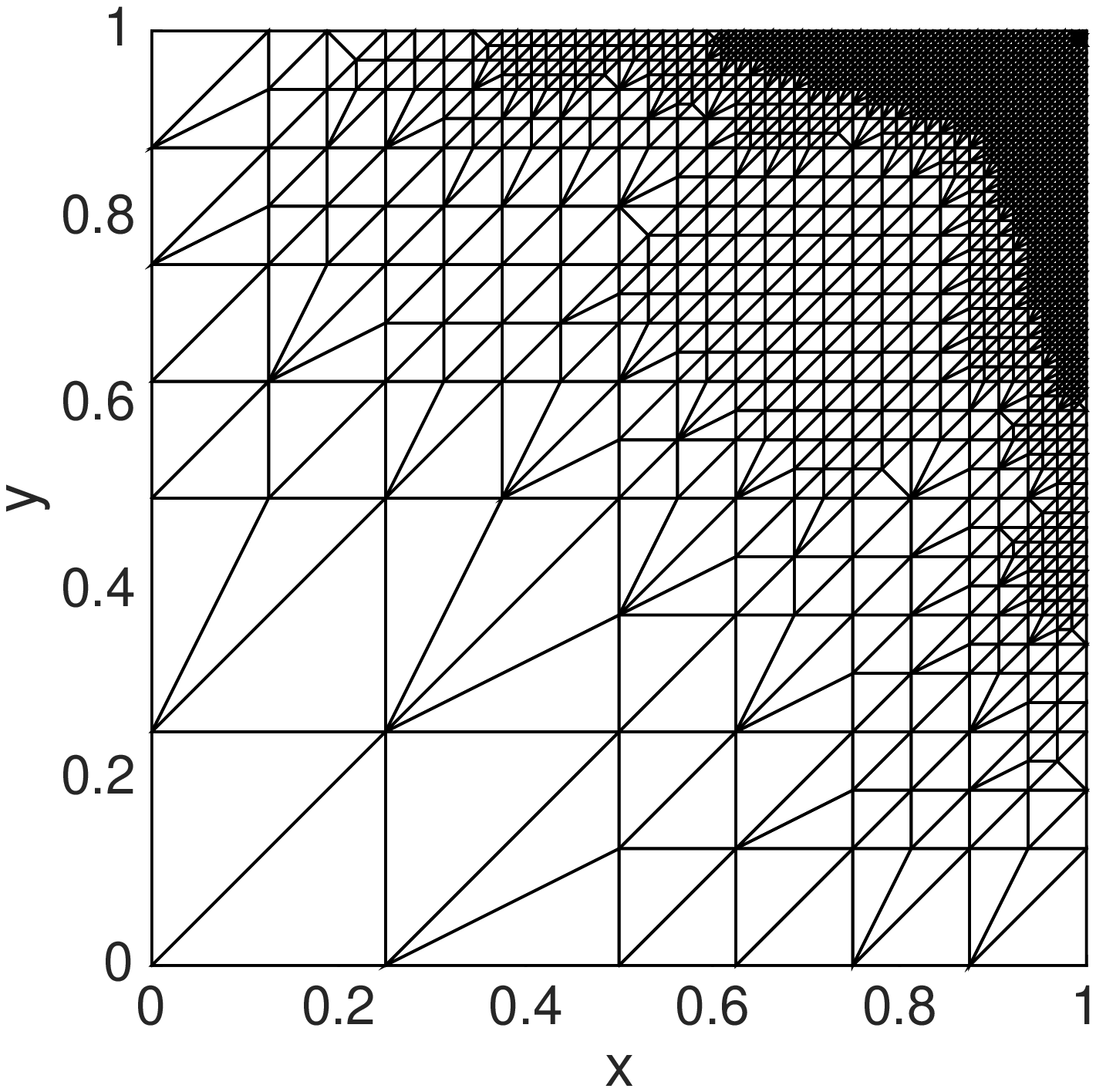}}
            \caption{Example 2. \protect\subref{fig:adaptive:reac_diff:err} Error in $\tnorm{\cdot}$-norm and error bound from
                \thref{thm:a posteriori} after the final fixed point iteration on each mesh compared to the number of degrees of freedom;
                \protect\subref{fig:adaptive:reac_diff:eff} Effectivity at each fixed point iteration for all meshes;
                \protect\subref{fig:adaptive:reac_diff:itn} Number of fixed point iterations on each mesh;
                \protect\subref{fig:adaptive:reac_diff:mesh} Mesh $\mesh[h,7]$ after 7 $h$-refinements.}
            \label{fig:adaptive:reac_diff}
        \end{figure}
        
        \emph{Example 2. Strong nonlinear reaction}
        We now consider a fairly strong nonlinear $f$ with a constant diffusion coefficient $\mu(\abs{\nabla u})=\varepsilon$, where $\varepsilon$ is a small positive constant,
        on the unit square $\Omega=(0,1)^2 \subset \real^2$. To this end, we consider $\varepsilon=0.01$ and
        let 
        \[
            f(u)=(0.2+x^2+y^2)\left(\frac{u^3}{u^2+1}+u\right)+c(x,y),
        \]
        where $c(x,y)$ is a function dependent only on $x$ and $y$ selected such that the analytical solution to
        \eqref{eqn:nonlinear PDE}--\eqref{eqn:nonlinear PDE BC} is given by
        \begin{equation}
            u(x, y) = (1-x)(1-y)(\e^{5x^2}-1)(\e^{5y^2}-1). \label{eqn:reac_diff:anal}
        \end{equation}
        We note that $\alpha_1=\alpha_2=\varepsilon$, $\beta_1 = \nicefrac{187}{40}$, and $\beta_2 = \nicefrac{1}{5}$. For this problem we set
        the steering parameter $\vartheta = 1$ in Algorithm~\ref{algo:refinement}.
              
        We again plot, in \subfigref{fig:adaptive:reac_diff}{err}, the relative true error $\nicefrac{\tnorm*{u-u_{h,i}^{n^\star}}}{\tnorm{u}}$ and 
        error bound, with $C_I=1$, from \thref{thm:a posteriori} after the last iteration $n^\star$ on each mesh $i$ against the number of
        degrees of freedom on that mesh.
        Except for a few early meshes, where the mesh is unlikely to accurately capture the
        boundary layer in the analytical solution, both the true error and the error bound converge at the similar
        rate with the error bound overestimating the true error by a roughly constant amount. This is supported by
        the effectivity indices at each iteration, \subfigref{fig:adaptive:reac_diff}{eff}, which are roughly constant
        for all meshes and iterations, although slightly decreasing over the course of the mesh refinement. For this problem
        we note that the number of fixed point iterations at each mesh step is fairly high, caused by the stronger nonlinearity,
        but only once the boundary layer is captured accurately by the mesh (the 5-th mesh onwards); cf.
        \subfigref{fig:adaptive:reac_diff}{itn}. Before this mesh the finite element error is considerably larger than the fixed point
        error due to the inaccurate capture of the boundary layer. The mesh after 7 $h$-adaptive mesh refinements,
        \subfigref{fig:adaptive:reac_diff}{mesh}, demonstrates how the mesh captures the boundary layer.
        This demonstrates
        an important benefit of only refining the fixed point error while it is greater than the finite element error, as
        the algorithm has managed to reduce the number of iterations in the early meshes by a considerable number
        by not performing fixed point iterations until after the mesh has started to accurately capture the solution's features.
        \medskip
        
        \begin{figure}[t]
            \centering
            \subfloat[]{\label{fig:adaptive:ex2:err}\includegraphics[width=0.49\textwidth]{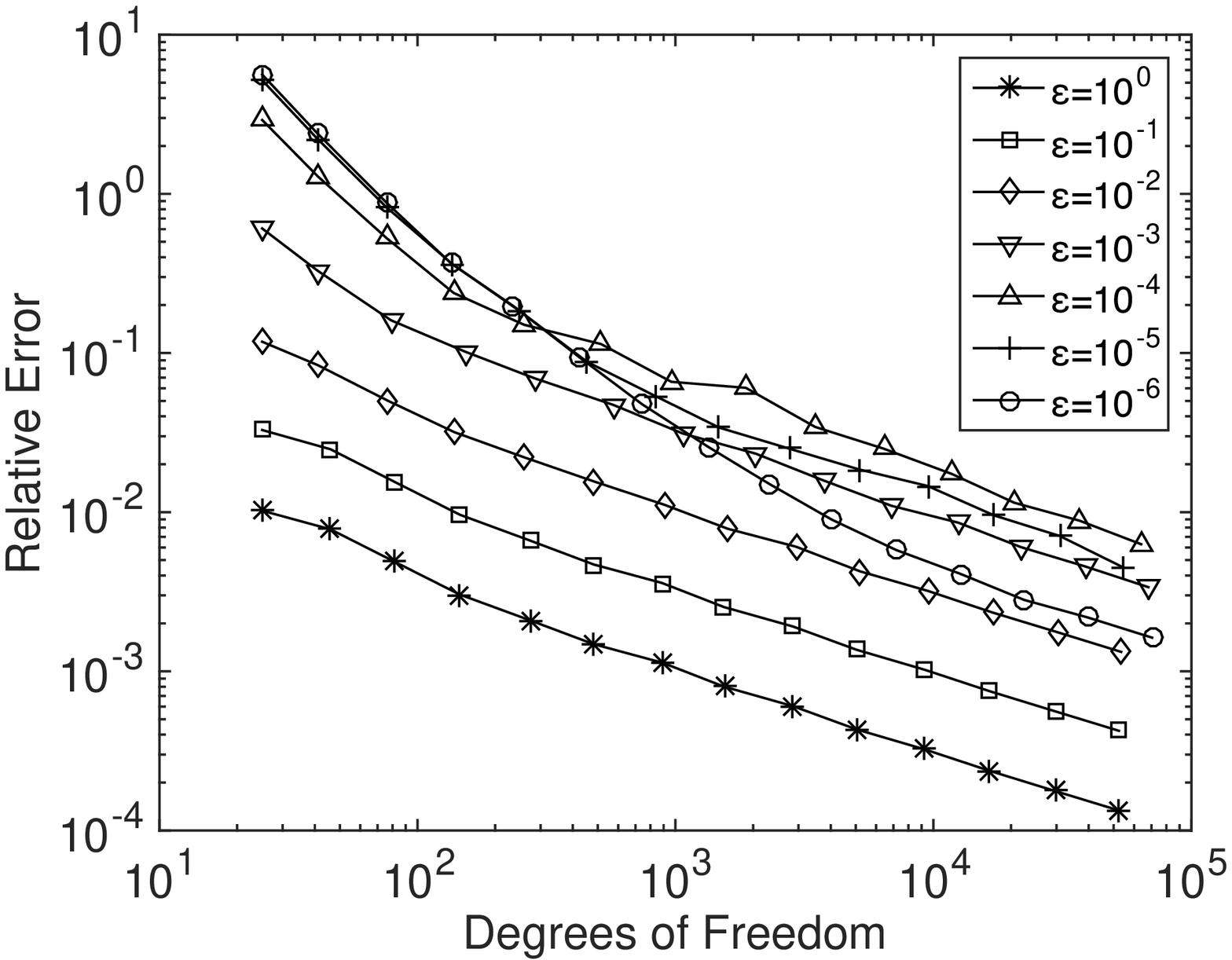}}
            \subfloat[]{\label{fig:adaptive:ex2:eff}\includegraphics[width=0.49\textwidth]{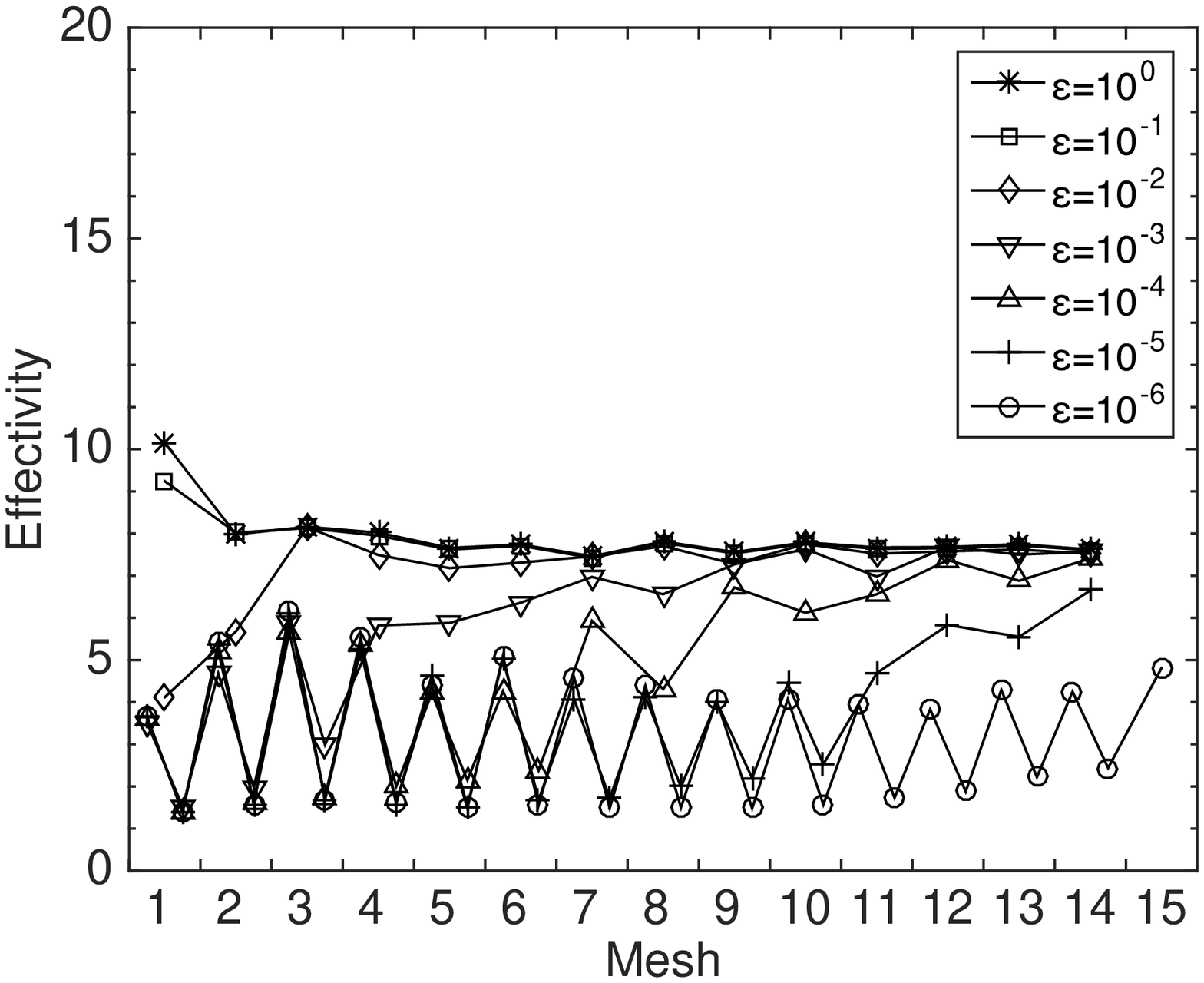}} \\
            \subfloat[]{\label{fig:adaptive:ex2:itn}\includegraphics[width=0.49\textwidth]{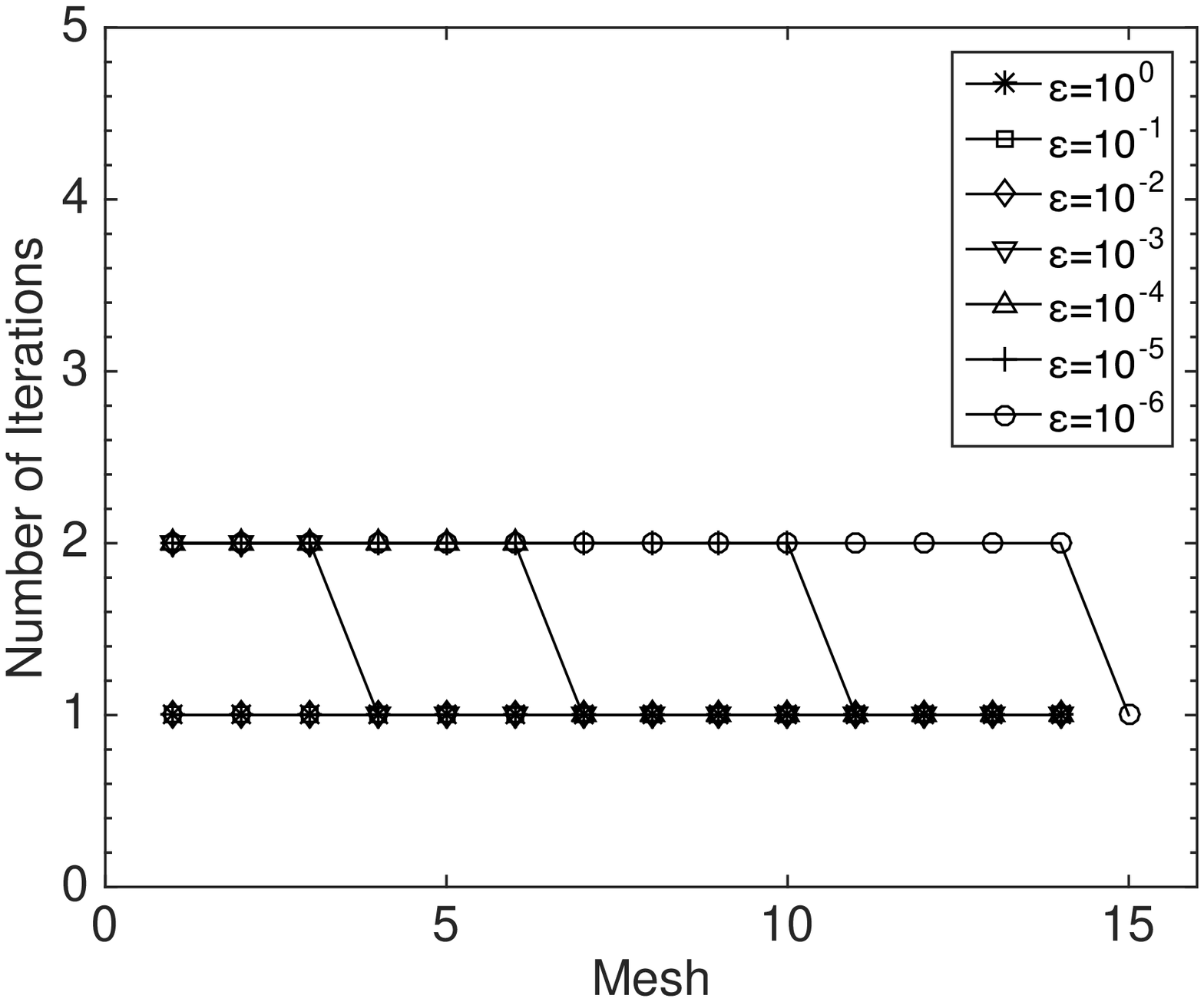}}
            \caption{Example 3. \protect\subref{fig:adaptive:ex2:err} Error in $\tnorm{\cdot}$-norm after the final fixed point iteration
                on each mesh compared to the number of degrees of freedom for various values of $\varepsilon$;
                \protect\subref{fig:adaptive:ex2:eff} Effectivity at each fixed point iteration for all meshes;
                \protect\subref{fig:adaptive:ex2:itn} Number of fixed point iterations on each mesh.}
            \label{fig:adaptive:ex2}
        \end{figure} 
        
        \emph{Example 3. Nonlinear reaction}
        We now consider a weaker nonlinear $f$ with a constant diffusion coefficient $\mu(\abs{\nabla u})=\varepsilon$, where $\varepsilon$ is a small positive constant, such that $\beta_1\approx\beta_2=\mathcal{O}(1)$,
        on the unit square $\Omega=(0,1)^2 \subset \real^2$. To this end, we consider $\varepsilon=10^{-k}$, for $k=0,\dots,6$, and
        let 
        \[
            f(u)=\frac{u^3}{10u^2+1}+u+c(x,y),
        \]
        where $c(x,y)$ is a function dependent only on $x$ and $y$ selected such that the analytical solution to
        \eqref{eqn:nonlinear PDE}--\eqref{eqn:nonlinear PDE BC} is given by \eqref{eqn:reac_diff:anal}.
        We note that $\alpha_1=\alpha_2=\varepsilon$, $\beta_1 = \nicefrac{89}{80}$, and $\beta_2 = 1$. For this problem we set
        the steering parameter $\vartheta = 1$ in Algorithm~\ref{algo:refinement}.       
       
        We plot, in \subfigref{fig:adaptive:ex2}{err}, the true error $\tnorm*{u-u_{h,i}^{n^\star}}$ for $\varepsilon=10^{-k}$, where
        $k=0,\dots,6$, on each mesh $i$ against the number of
        degrees of freedom on that mesh. We note that we appear to achieve a higher initial rate of convergence for smaller $\varepsilon$ values,
        although they all appear to tend to similar convergence rates as refinement progresses.  
        For each $\varepsilon$ we also calculate the effectivity indices at each iteration,
        \subfigref{fig:adaptive:ex2}{eff}. We note that initially they are highly oscillatory for small values of~$\varepsilon$, but as refinements progress they tend to smoothen towards a constant value, with the high values of $\varepsilon$ converging at earlier mesh steps; in particular, the effectivity indices do not deteriorate as~$\varepsilon\to 0^+$. We also plot in
        \subfigref{fig:adaptive:ex2}{itn} the number of fixed point iterations at each mesh step required to ensure that
        the fixed point error is less than the finite element error. We note this is fairly constant and independent of the value
        of $\varepsilon$, which supports Remark~\ref{rm:epsilon}.

\section{Conclusion}
\label{sec:conclusion}
    In this article we have shown that it is possible within a finite element framework to use a simple fixed point
    iteration to solve strongly monotone quasi-linear elliptic PDEs requiring only the computation of the iteration matrix,
    opposed to a Newton's method requiring computation at each iteration. We have shown that an optimal \emph{a priori}
    convergence rate can be obtained for a fixed number of iterations, dependent on the mesh size or polynomial degree.
    We have demonstrated that it is possible to perform adaptive mesh refinement based on an \emph{a posteriori} error analysis, where only
    a minimal number of fixed point iterations are required on each mesh to obtain a good approximation to the solution, without
    continuing the fixed point iteration until the fixed point error is insignificant.
\bibliographystyle{amsplain}
\bibliography{references}

\end{document}